\documentclass[preprint,10pt]{elsarticle}

\usepackage{amsmath,amsopn,amssymb,epsfig}

\newtheorem{Algorithm}{Algorithm}[section]
\newtheorem{Theorem}{Theorem}[section]
\newtheorem{Lemma}{Lemma}[section]
\newtheorem{Proposition}{Proposition}[section]
\newtheorem{Remark}{Remark}[section]

\newtheorem{example}{Example}[section]

\newproof{proof}{Proof}
\newproof{pot}{Proof of Theorem \ref{thm2}}
\newcommand{\bb}{\begin{bmatrix}}
\newcommand{\eb}{\end{bmatrix}}
\newcommand{\bl}[1]{\begin{list}{#1}{\usecounter{bean}}} \newcommand{\el}{\end{list}}
\newcommand{\bel}[1]{\begin{equation} \label{#1}} \newcommand{\eel}{\end{equation}}

\def\r2n2n{\mathbb{R}^{2n\times 2n}}
\def\c2n2n{\mathbb{C}^{2n\times 2n}}

\begin{document}

\date{}
\begin{frontmatter}
\title{
On a class of nonlinear matrix equations $X\pm A^{\small H}f(X)^{-1}A=Q$
\tnoteref{t1}}
\author{Chun-Yueh Chiang\corref{cor1}\fnref{fn1}}
\ead{chiang@nfu.edu.tw}
\address{Center for General Education, National Formosa
University, Huwei 632, Taiwan.}

\cortext[cor1]{Corresponding author}
\fntext[fn1]{The author was supported
 by the Ministry of Science and Technology of Taiwan under grant
MOST 105-2115-M-150-001.}

\date{ }

\begin{abstract}
Nonlinear matrix equations are encountered in many
applications of control and engineering problems. In this work, we establish a complete study
for a class of nonlinear matrix equations. With the aid of Sherman Morrison Woodbury formula,
  we have shown that any equation in this class has the maximal positive definite solution under a certain condition. Furthermore, A thorough study of properties about this class of matrix equations is provided. An acceleration of iterative method with R-superlinear convergence with order $r>1$ is then designed to solve the maximal positive definite solution efficiently.
\end{abstract}

\begin{keyword}
Nonlinear matrix equation,\,Sherman Morrison Woodbury formula,\,Maximal positive definite solution,\,Flow,\,Positive operator,\,Doubling algorithm,\, R-superlinear with order $r$
\MSC 65F10 \sep  65H05 \sep  15A24\sep 15A86
\end{keyword}

\end{frontmatter}
\section{Introduction}
In the paper we consider a class of nonlinear matrix equations (NMEs) with the plus sign
\begin{subequations}\label{eq:NME}
\begin{align}\label{eq:NMEP}
X+A^{H}f(X)^{-1}A=Q
\end{align}
and the minus sign
\begin{align}\label{eq:NMEM}
X-A^{H}f(X)^{-1}A=Q,
\end{align}
\end{subequations}
where $A \in \mathbb{F}^{n \times n}$, 
  $Q$ is a Hermitian ( or symmetric) matrix with size $n\times n$, and the $n$-square matrix $X$ is an unknown Hermitian matrix and will be determined. The base field $\mathbb{F}$ can be the real field $\mathbb{R}$ and complex field $\mathbb{C}$.
The transformation $f:\mathbb{F}^{n\times n}\rightarrow \mathbb{F}^{n\times n}$ is a
matrix operator satisfying the following suitable assumptions:
\begin{itemize}
\item[(a1)]Period-$2$: $f^{(2)}(A):=f(f(A))=A,\, A\in\mathbb{F}^{n\times n}$.
\item[(a2)]Preserve additions (or additivity): $f(A+B)=f(A)+f(B),\, A,\,B\in\mathbb{F}^{n\times n}$.
\item[(a3)]Preserve products: Multiplication order preserving if $f(AB)=f(A)f(B)$ or multiplication order reversing if $f(AB)=f(B)f(A)$,\, $A,\,B\in\mathbb{F}^{n\times n}$.
\item[(a4)]Preserve nonnegativity: $f(A)\geq 0$ whenever $A\geq 0$. Here, we use the usual partial order for Hermitian matrices, i.e.,  $X > Y (X \geq Y )$ if $X-Y$ is positive definite (semidefinite) for two Hermitian matrices $X$ and $Y$.
\end{itemize}
We begin with a brief study of some basic properties of operator $f$.
Some interesting features about $f$ are given here.
\begin{Proposition}\label{pro0}
\par\noindent
\begin{itemize}
\item[(1)] $f(rX)=rf(X)$ for any rational number $r$ and every $X$ in $\mathbb{F}^{n\times n}$. Moreover,
$f$ is a continuous operator on $\mathbb{F}^{n\times n}$ if $f$ is continuous at $X=0$.
\item[(2)] $f$ is said to be unital, i.e., $f(I_n)=I_n$. And $A$ is nonsingular if and only if $f(A)$ is nonsingular. Moreover, $f(A^{-1})=f(A)^{-1}$ whenever $A$ is invertible.
\item[(3)] $f$ is the so-called order preserving operator on $\mathbb{H}^{n\times n}$, i.e., $f(A)\geq f(B)$ ($f(A)\geq f(B)$) for $A\geq B$ ($A> B$), $A,B\in\mathbb{H}^{n\times n}$, where $\mathbb{H}^{n\times n}$ is the set of all $n\times n$ Hermitian matrices.
\item[(4)] $f$ is adjoint-preserving, i.e., $f(A^H)=f(A)^H$
for all $A\in\mathbb{F}^{n\times n}$.
\end{itemize}
\end{Proposition}
\begin{proof}
\par\noindent
\begin{itemize}
\item[1.] Let $X\in\mathbb{F}^{n\times n}$ and $r=m/n$, where $m$ and $n$ are two integers with $n\neq 0$.
It can be easily seen that $f(0_{n\times n})=0_{n\times n}$ and $f(-X)=-f(X)$. Without loss of generality we assume that $m$ and $n$ are positive. Clearly, $nf(rX)=f(nrX)=f(mX)=mf(X)$ and therefore $f(rX)=rf(X)$. The remaining part follows directly from the definition of a continuous operator.
\item[2.] Let $A\in\mathbb{F}^{n\times n}$ be a nonsingular matrix. Since $A=f(f(A)I_n)=Af(I_n)$ or $f(I_n)A$, thus $f(I_n)=I_n$. It is trivial that $f(A)$ is nonsingular and $f(A)=f(A^{-1})^{-1}$. This completes the proof.
\item[3.] In the case of \text{``$\geq$''}, the result is immediately clear from the definition of order preserving operator. The result is valid in the case of \text{``$>$''} since $A-B$ is nonsingular.
\item[4.] First, for a Hermitian matrix $A$ it is easily seen that $A$ can be decomposed as $A=A_1-A_2$ with two positive matrices $A_1$ and $A_2$ and thus $f(A^H)=f(A_1)-f(A_2)=f(A)^H$. Now, for an arbitrary complex matrix $A$ we have the so-called Cartesian decomposition~\cite{Bernstein2005}
    \begin{align*}
    A=H_1+i H_2,
    \end{align*}
    where $H_1$ and $H_2$ are two Hermitian matrices. The result will be verified by showing that
    \[
     f(-iH_2)= f(iH_2)^H.
    \]
Indeed, $f(iH_2)^2=f(-H_2^2)<0$. We conclude that $f(iH_2)$ is unitarily equivalent to a diagonal matrix diag$(\pm i\sqrt{h_1},\cdots,\pm i\sqrt{h_n})$, where $h_i$ are real positive eigenvalues of $-f(iH_2)^2$, $i=1,\cdots, n$. The result is proved.
\end{itemize}
\end{proof}

Clearly, for any $X \in\mathbb{F}^{n\times n}$  the identity operator $f(X):=X$, the transpose operator $f(X):=X^\top$ if $\mathbb{F}=\mathbb{R}$, and the conjugate operator $f(X):=\overline{X}$ if $\mathbb{F}=\mathbb{C}$ all of these satisfy conditions (a1)--(a4). Let $\mathbb{F}=\mathbb{R}$ and $f$ be a bijective linear operator. It follows from \cite{emrl2006364} that there exists a invertible matrix $T_f$ such that
\begin{itemize}
\item[a.] $f(X)=T_fXT_f^{-1}$ if $f$ is multiplication preserving,
\end{itemize}
or
\begin{itemize}
\item[b.] $f(X)=T_fX^\top T_f^{-1}$ if $f$ is multiplication reversing,
\end{itemize}
for all $X\in\mathbb{F}^{n \times n}$. See \cite{emrl2006364,Fang2006601} for details.
Therefore, a bijective continuous operator $f$ satisfying (a3) can be reduced to the identity or transpose operator $f$ in the real field.

NMEs like the form~\eqref{eq:NME} occur frequently in many applications, that
include control theory, ladder networks, dynamic programming,
stochastic filtering and statics when $f$ is the identity operator \cite{Anderson90,Zhan_96_SISC}.
Notable examples include algebraic Riccati equations \cite{Chu05,Chu04,Hwang05,Hwang07,Lin06}.
The main application of Eq.~\eqref{eq:NMEP} with conjugate operator $f$ arises from the study of consimilarity . For more detail of application of NMEs, see \cite{Li2014546,Zhou20137377}. In this paper, we are interested in the study of the positive definite solutions of Eqs.~\eqref{eq:NME} and its solvable condition. By making use of Sherman Morrison Woodbury formula~\cite{Bernstein2005}, first we propose a kind of fixed-point iterative method
\begin{subequations}\label{fix0}
\begin{align}
X_k&=F({X}_{k-1},{X}_{1}),\\
 {X}_1&{\mbox{ is given}},
\end{align}
\end{subequations}
for finding the maximal positive definite solutions of Eqs.~\eqref{eq:NME}. As is well-known in
 the study of ordinary differential equations, let $x(t)$ be a solution of the autonomous system,
\begin{subequations}\label{ivp}
\begin{align}
 \dot{x}&=F(x(t)),\, t\geq0,\\
 x(0)&=x_0{\mbox{ is given}},
\end{align}
\end{subequations}
and we define a flow $\phi$ which is the mapping $\phi_{x_0}(t):=x(t)$.
 Then, for any $s,t\geq0$ the flow $\phi$ satisfies the following \emph{group law}
 \[
 \phi_{x_0}(s+t)=\phi_{x_s}(t)=\phi_{x_t}(s)
 \]
 if the problem~\eqref{ivp} is uniquely solvable for any initial value $x(0)$. Similar to the group law of the phase flow of a differential equation with an initial value, let the ordinary difference equation be
 \begin{subequations}\label{ivp1}
\begin{align}
 x(k+1)&=F(x(k)),\, k\geq1,\\
 x(1)&=x_1{\mbox{ is given}},
\end{align}
\end{subequations}
associated with the iterated function $F$ and a starting value $x_1$. We also have
 \begin{align}\label{group0}
 \phi_{x_1}(s+t-1)=\phi_{x_s}(t)=\phi_{x_t}(s),\, s,t\geq 1,
 \end{align}
 where the iteration solution $\phi_{x_1}(n):=x(n)$ for the problem \eqref{ivp1}. However, the group law \eqref{group0}
 is not always guaranteed if $F$ is associated with the initial value $x_1$.
 \begin{example}
  Let  $x_k=f(x_{k-1},x_1)=ax_{k-1}+x_1$, where $a$ and $x_1$ are given as two nonzero real numbers. Then, $\phi_{x_1}(k)=\sum\limits_{i=0}^{k-1} a^i x_1$.
  Clearly,  $\phi_{x_1}(s+t-1)\neq \phi_{x_s}(t)$ for $s=2$, $t=1$ and $a\neq-1$.
  \end{example}
The group law \eqref{group0} plays an important role in the techniques to study an efficient iterative
algorithm for solving the solutions of some nonlinear matrix equations \cite{ChiangThesis08,Bini2012}. In the paper, we derive a property similar to the group law to the iteration \eqref{fix0}. With the help of this property, two accelerated iterative methods for solving the maximal positive definite solution of Eqs.~\eqref{eq:NME} are developed based on \eqref{fix0}. In addition, many elegant properties of this kind of iteration~\eqref{fix0} will be established.

 We introduce the following well-known results which we need in the rest of the paper. The results in the following lemma either follow immediately from the definition or are easy to verify.
\begin{Lemma}\cite{Bernstein2005}\label{Schur}
Let $A$ be an arbitrary matrix of size $n$. $X$ and $Y$ are two $n\times n$ positive definite matrices. Then,
\begin{itemize}
\item[1.][Sherman Morrison Woodbury formula (SMWF)] Assume that $Y^{-1}\pm AX^{-1}A^H$ is nonsingular. Then, $X\pm A^H Y A$ is invertible and
\[
(X\pm A^H Y A)^{-1}=X^{-1}\mp X^{-1}A^H(Y^{-1}\pm AX^{-1}A^H)^{-1}AX^{-1}.
\]
\item[2.][Schur complement] A square complex matrix $\Psi$ is partitioned as
\[
\Psi:=\bb X & A \\ A^H & Y\eb.
\]
Then, $\Psi>0$ ($\Psi\geq 0$) if and only if $Y-A^H X^{-1} A>0$ ($Y-A^H X^{-1} A\geq0$) if and only if $X-A Y^{-1}A^H>0$ ($X-A Y^{-1}A^H\geq0$).
\item[3.]$X>Y$ ($X\geq Y$) if and only if  $\rho(YX^{-1})=\rho(X^{-1}Y)<1$ ($\rho(X^{-1}Y)\leq 1$), where $\rho(X)$ denotes the spectral radius of $X$.
\end{itemize}
\end{Lemma}

Without loss of generality we assume that $f$ is multiplication order preserving throughout the paper. Let $\|.\|$ denote a norm on $\mathbb{F}^{n\times n}$ as well as the induced matrix norm.
  We say that a sequence of matrices $\{X_n\}$ converges R-linearly to $X$ if
\begin{align}\label{q1}
\|X_{k+1}-X\|<c\sigma^{k},
\end{align}
and converges R-superlinearly to $X$ with order $r$ if
\begin{align}\label{q2}
\|X_{k+1}-X\|<c\sigma^{r^k},
\end{align}
for arbitrary positive integer $k$,  where $c>0$, $0<\sigma<1$, and $r$ is an integer greater than $1$. The quantity $\sigma$ is called the convergence rate of this sequence. Especially, $X_k$ converges to $X$ is said to be R-sublinearly, R-quadratically and R-cubically if
$X_k\rightarrow X$ as $k\rightarrow\infty$ and $\sigma=1$ in \eqref{q1}, $r=2$ in \eqref{q2}, $r=3$ in \eqref{q2}, respectively, see~\cite{doi:10.1137/1.9780898719468,Kelley95,Bini2012} and the references therein.
  A positive definite solution $X_M$ of NME~\eqref{eq:NMEP} (or \eqref{eq:NMEM}) is called maximal (or minimal) if $X_M\geq X$ (or $X_M\leq X$) for any symmetric solution $X$ of Eq.~\eqref{eq:NMEP} (or \eqref{eq:NMEM}). The symbol $\mathbb{P}^{n\times n}$ stand for the set of $n\times n$ positive definite matrices. We denote the $m\times m$ identity matrix by $I_m$, the conjugate transpose matrix of $A$ by $A^H$, the spectrum of $A$ by $\sigma(A)$ and use $\det(A)$ to denote the determinant of a square matrix $A$. Given a matrix operator $\mathcal{F}$, $\mathcal{F}^{(i)}$ denote the composition of $\mathcal{F}$ with itself $i$ times and  $\mathcal{F}^{(0)}:=I$ is the identity map.

 The paper is organized as follows. In Sections~2 and 3, we describe how to transform Eqs.~\eqref{eq:NME} to a standard nonlinear matrix equation and provide a fixed-point iteration to compute the maximal positive definite solution. Moreover, we formulate the necessary and sufficient conditions for the existence of the maximal positive definite solution of Eq.~\eqref{eq:NMEP} or Eq.~\eqref{eq:NMEM} directly by means of the solvable analysis of the standard nonlinear matrix equation. A R-superlinearly convergent iterative method with order $r>1$ is briefly discussed in Section 4. Also, in Section~5 an alternating iteration gives the same result as considering a different substraction. Finally, concluding remarks are given in Section 6.

\section{The transformation technique based on SMWF}
In \cite{Chiang2015}, we recently proposed a standard way to find a sufficient condition for the unique solvability of a class of Sylvester equations. A useful method to investigate a famous matrix equation is to simplify it by applying suitable transformations to the unknowns or to the coefficient matrices. There, the goal was to analyze Eqs.~\eqref{eq:NME} with the help of some well-known matrix equations.

To facilitate our discussion, we first transform Eqs.~\eqref{eq:NME}  into the standard nonlinear matrix equation after one step of SMWF:
\begin{align}\label{NMEP}
X+(A^{(1)})^H (X-B^{(1)})^{-1} A^{(1)}=Q^{(1)},
\end{align}
where the initial matrices
\begin{subequations}\label{initial2}
\begin{align}\label{initial}
A^{(1)}:=f(A)f(Q)^{-1}A,\, B^{(1)}:=f(A)f(Q)^{-1}f(A)^H,\,\mbox{and}\quad  Q^{(1)}:=Q-A^H f(Q)^{-1}A
\end{align}
corresponding to the Eq.~\eqref{eq:NMEP} and the initial matrices
\begin{align}\label{initial1}
A^{(1)}:=f(A)f(Q)^{-1}A,\, B^{(1)}:=-f(A)f(Q)^{-1}f(A)^H,\,\mbox{and}\quad  Q^{(1)}:=Q+A^H f(Q)^{-1}A
\end{align}
corresponding to the Eq.~\eqref{eq:NMEM}.
\end{subequations}
For the sake of simplicity let the matrix operators $\mathcal{F}_{+}$ and $\mathcal{F}_{-}$ be defined by
\begin{align}\label{eq:NME1}
X=\mathcal{F}_{\pm}(X):=Q\mp A^{H}f(X)^{-1}A,
\end{align}
respectively,
and let $\mathcal{G}_{\pm}$ be defined by
\begin{align}\label{eq:dualNME}
\widetilde{X}=\mathcal{G}_{\pm}(\widetilde{X}):=Q\mp f(A)f(\widetilde{X})^{-1}f(A)^{H}.
\end{align}
Clearly, $\mathcal{F}_+$ and $-\mathcal{F}_-$  are order preserving for any positive integer $k$.
For convenience we adopt the notation $\mathcal{F}(X)$ by $Q^{(1)}-(A^{(1)})^H (X-B^{(1)})^{-1} A^{(1)}$,
repeating the same argument gives that
\begin{align}\label{kNME}
X&=\mathcal{F}^{(k)}(X):=Q^{(k)}-(A^{(k)})^H (X-B^{(k)})^{-1} A^{(k)},
\end{align}
 where the sequences of matrices $\{A^{(k)}\}$,  $\{B^{(k)}\}$ and $\{Q^{(k)}\}$ are generated (if no breakdown occurs) by
\begin{subequations}\label{signal}
\begin{align}
A^{(k)}&:=A^{(1)}(Q^{(1)}-B^{(k-1)})^{-1}A^{(k-1)},\label{signalA}\\
B^{(k)}&:=B^{(1)}+A^{(1)}(Q^{(1)}-B^{(k-1)})^{-1}(A^{(1)})^H,\label{signalB}\\
Q^{(k)}&:=Q^{(k-1)}-(A^{(k-1)})^H (Q^{(1)}-B^{(k-1)})^{-1} A^{(k-1)}\label{signalQ},
\end{align}
\end{subequations}
for any positive integer $k>1$. Now, we study some characteristics about the iterated function~\eqref{signal}.
For simplicity's sake we define the symbol $\widehat{\mathbb{X}}^{(k)}:=(\widehat{A}^{(k)},\widehat{B}^{(k)},\widehat{Q}^{(k)})$ and we rewrite the iteration
\begin{align*}
\widehat{A}^{(k)}&:=A_2(Q_2-\widehat{B}^{(k-1)})^{-1}A^{(k-1)},\\
\widehat{B}^{(k)}&:=B_2+A_2(Q_2-\widehat{B}^{(k-1)})^{-1}(A_2)^H,\\
\widehat{Q}^{(k)}&:=\widehat{Q}^{(k-1)}-(\widehat{A}^{(k-1)})^H (Q_2-\widehat{B}^{(k-1)})^{-1} \widehat{A}^{(k-1)},
\end{align*}
with initial matrices $\mathbb{X}_1:=(A_1,B_1,Q_1)$ and constant matrices $\mathbb{X}_2:=(A_2,B_2,Q_2)$ as the notation
\begin{align}\label{iteration}
\widehat{\mathbb{X}}^{(k)}=\mathcal{I}_{\mathbb{X}_1}(\widehat{\mathbb{X}}^{(k-1)},\mathbb{X}_2),
\end{align}
 where the iterated function $\mathcal{I}_{\mathbb{X}_1}:\mathbb{F}^n\times\mathbb{H}^n\times\mathbb{H}^n\rightarrow \mathbb{F}^n\times\mathbb{H}^n\times\mathbb{H}^n$ presents the relationship between $(\widehat{A}^{(k-1)},\widehat{B}^{(k-1)},\widehat{Q}^{(k-1)})$ and $(\widehat{A}^{(k)},\widehat{B}^{(k)},\widehat{Q}^{(k)})$.  Especially, we denote the iteration \eqref{signal} with  initial matrices $\mathbb{X}^{(1)}=(A^{(1)},B^{(1)},Q^{(1)})$ by $\mathbb{X}^{(k)}=\mathcal{I}_{\mathbb{X}^{(1)}}(\mathbb{X}^{(k-1)},\mathbb{X}^{(1)})=(A^{(k)},B^{(k)},Q^{(k)})$.
 Now, the following fundamental \emph{group-like law} holds and would provide a great advantage for emerging some numerical algorithms.
\begin{Proposition}\label{trans}
 Suppose that all sequences of matrices generated by iterations \eqref{signal} with initial matrices \eqref{initial2} are no breakdown. Assume that $\widehat{\mathbb{X}}^{(k)}=\mathcal{I}_{\mathbb{X}^{(i)}}(\widehat{\mathbb{X}}^{(k-1)},\mathbb{X}^{(j)})$ is well-defined for any integers $i,j\geq 1$ and $k>1$. We conclude that the sequence $\{\mathbb{X}^{(i)}\}$ satisfies
\begin{align*}
\mathbb{X}^{(i+j)}=\widehat{\mathbb{X}}^{(2)}.
\end{align*}
Namely, we have
\begin{subequations}\label{tran1}
\begin{align}
A^{(i+j)}&=A^{(j)}(Q^{(j)}-B^{(i)})^{-1}A^{(i)},\label{tran1a}\\
B^{(i+j)}&=B^{(j)}+A^{(j)}(Q^{(j)}-B^{(i)})^{-1}(A^{(j)})^H,\label{tran1b}\\
Q^{(i+j)}&=Q^{(i)}-(A^{(i)})^H (Q^{(j)}-B^{(i)})^{-1} A^{(i)}.\label{tran1q}
\end{align}
\end{subequations}
Furthermore, we have
\begin{align}\label{ijk}
\mathbb{X}^{(i+(k-1)j)}&=\widehat{\mathbb{X}}^{(k)},
\end{align}
where $k$ is any positive integer.
\end{Proposition}
\begin{proof}
First, let $i$ and $j$ be any two integers and $\Delta_{j,i}:=(Q^{(j)}-B^{(i)})^{-1}$.
For each $j$, we will prove \eqref{tran1} by the principle of mathematical induction with respect to $i$.
The proof is divided into two parts,
\begin{enumerate}
 \item For $i=1$, we show that
\begin{align*}
A^{(1+j)}&=A^{(j)}\Delta_{j,1}A^{(1)},\\
B^{(1+j)}&=B^{(j)}+A^{(j)}\Delta_{j,1}(A^{(j)})^H,\\
Q^{(1+j)}&=Q^{(1)}-(A^{(1)})^H \Delta_{j,1} A^{(1)},
\end{align*}
 by using induction. For $j=1$ it is trivial from the definition of $A^{(2)}$, $B^{(2)}$ and $Q^{(2)}$. Now suppose that it is true for $j=s$. Note that
\begin{align*}
\Delta_{1,s+1}&=\Delta_{1,s}+\Delta_{1,s}A^{(s)}\Delta_{s+1,1}(A^{(s)})^H\Delta_{1,s},\\
\Delta_{s+1,1}&=\Delta_{s,1}+\Delta_{s,1}(A^{(1)})^H\Delta_{1,s+1}A^{(1)}\Delta_{s,1}.
\end{align*}
Then,
\begin{align*}
A^{(s+2)}&=A^{(1+s+1)}=A^{(1)} \Delta_{1,s+1} A^{(s+1)}\\
&=A^{(1)}\left[\Delta_{1,s}+\Delta_{1,s}A^{(s)}\Delta_{s+1,1}(A^{(s)})^H\Delta_{1,s}\right]A^{(s)}\Delta_{s,1}A^{(1)}\\
&=A^{(s+1)}\left[\Delta_{s,1}+\Delta_{s+1,1}(A^{(s)})^H\Delta_{1,s}A^{(s)}\Delta_{s,1}\right]A^{(1)}\\
&=A^{(s+1)}\Delta_{s+1,1}\left[Q^{(s+1)}-B^{(1)}+(A^{(s)})^H\Delta_{s,1}A^{(s)}\right]\Delta_{s,1}A^{(1)}\\
&=A^{(s+1)}\Delta_{s+1,1}A^{(1)},
\end{align*}

\begin{align*}
B^{(s+2)}&=B^{(1+s+1)}=B^{(1)}+A^{(1)} \Delta_{1,s+1} (A^{(1)})^H\\
&=B^{(1)}+A^{(1)}\left[\Delta_{1,s}+\Delta_{1,s}A^{(s)}\Delta_{s+1,1}(A^{(s)})^H\Delta_{1,s}\right](A^{(1)})^H\\
&=B^{(s+1)}+A^{(s+1)} \Delta_{s+1,1} (A^{(s+1)})^H
\end{align*}
and
\begin{align*}
Q^{(s+2)}&=Q^{(1+s+1)}=Q^{(s+1)}-(A^{(s+1)})^H\Delta_{1,s+1}A^{(s+1)} \\
&=Q^{(1)}-(A^{(1)})^H \Delta_{s,1} A^{(1)}-(A^{(1)})^H\Delta_{s,1}(A^{(s)})^H\Delta_{1,s+1}A^{(s)}\Delta_{s,1}A^{(1)}\\
&=Q^{(1)}-(A^{(1)})^H\Delta_{s+1,1}A^{(1)}.
\end{align*}
 The result is proved.
\item Now suppose that \eqref{tran1} is true for $i=s$ and any $j$.
 Note that
\begin{align*}
\Delta_{j,s+1}&=\Delta_{j,s}+\Delta_{j,s}A^{(s)}\Delta_{s+j,1}(A^{(s)})^H\Delta_{j,s},\\
\Delta_{s+j,1}&=\Delta_{s,1}+\Delta_{s,1}(A^{(s)})^H\Delta_{j,s+1}A^{(s)}\Delta_{s,1}.
\end{align*}
 Then,
 \begin{align*}
A^{(s+1+j)}&=A^{(1+s+j)}=A^{(s+j)} \Delta_{s+j,1} A^{(1)}\\
&=A^{(j)}\Delta_{j,s}A^{(s)} \left[\Delta_{s,1}-\Delta_{s,1}(A^{(s)})^H\Delta_{j,s+1}A^{(s)}\Delta_{s,1}\right]A^{(1)}\\
&=A^{(j)}\left[\Delta_{j,s}-\Delta_{j,s}A^{(s)} \Delta_{s,1}(A^{(s)})^H\Delta_{j,s+1}\right]A^{(s+1)}\\
&=A^{(j)}\Delta_{j,s}\left[Q^{(j)}-B^{(s+1)}-A^{(s)} \Delta_{s,1}(A^{(s)})^H\right]\Delta_{j,s+1}A^{(i)}\\
&=A^{(j)}\Delta_{j,s+1}A^{(s+1)},
\end{align*}
\begin{align*}
B^{(s+1+j)}&=B^{(1+s+j)}=B^{(s+j)}+A^{(s+j)} \Delta_{s+j,1} (A^{(s+j)})^H\\
&=B^{(j)}+A^{(j)}\left[\Delta_{j,s}+\Delta_{j,s}A^{(s)}\Delta_{s+j,1}(A^{(s)})^H\Delta_{j,s}\right](A^{(j)})^H\\
&=B^{(j)}+A^{(j)}\Delta_{j,s+1}(A^{(j)})^H,
\end{align*}
and
 \begin{align*}
Q^{(s+1+j)}&=Q^{(1+s+j)}=Q^{(1)}-(A^{(1)})^H\Delta_{s+j,1}A^{(1)}\\
&=Q^{(1)}-(A^{(1)})^H\Delta_{s,1}A^{(1)}-(A^{1})^H\left[\Delta_{s+j,1}-\Delta_{s,1}\right]A^{(1)}\\
&=Q^{(s+1)}-(A^{(1)})^H\Delta_{s,1}(A^{(s)})^H\Delta_{j,s+1}A^{(s)}\Delta_{s,1}A^{(1)}\\
&=Q^{(s+1)}-(A^{(s+1)})^H\Delta_{j,s+1}A^{(s+1)}.
\end{align*}
The induction process is now finished and thus the result is followed.
\end{enumerate}
Therefore, we have proved \eqref{tran1}. In addition, the formula \eqref{ijk} for any integer $k\geq 1$ can easily be proved by using induction. For $k=1$ it is clear that the formula holds. Assume that \eqref{ijk} is true for $k=s$ and any positive integers $i$ and $j$. Together with the group property we have
\begin{align*}
\widehat{\mathbb{X}}^{(s+1)}&=\mathcal{I}_{\widehat{\mathbb{X}}^{(1)}}(\widehat{\mathbb{X}}^{(s)},\mathbb{X}^{(i)})
=\mathcal{I}_{\widehat{\mathbb{X}}^{(s)}}(\widehat{\mathbb{X}}^{(1)},\mathbb{X}^{(i)})
=\mathcal{I}_{\mathbb{X}^{(i+(s-1)j)}}(\widehat{\mathbb{X}}^{(1)},\mathbb{X}^{(i)})=\mathbb{X}^{(i+sj)}.
\end{align*}
This completes the proof.
\end{proof}
\begin{Remark}\label{Rmk}
\par\noindent
\begin{itemize}
\item[1.] In the recursive algorithm~\eqref{signal}, the iteration $B^{(k)}$ is clear independent of the other two iterations. From Proposition~\ref{trans} we also have
    \begin{align}\label{fixQ}
    Q^{(k)}=Q^{(1)}-(A^{(1)})^H (Q^{(k-1)}-B^{(1)})^{-1} A^{(1)},
    \end{align}
    i.e., the iteration $Q^{(k)}$ is also independent of the other two iterations~\eqref{signalA} and ~\eqref{signalB} and is the fixed-point iteration of Eq.~\eqref{NMEP}. On the other hand, the $k$th iterations $Q^{(k)}$ and $B^{(k)}$  can be obtained from the following finite series form,
\begin{align*}
Q^{(k)}&=Q^{(1)}-\sum\limits_{i=1}^{k-1}  (A^{(i)})^H (Q^{(1)}-B^{(i)})^{-1} A^{(i)},\\
B^{(k)}&=B^{(1)}+\sum\limits_{i=1}^{k-1} A^{(i)}(Q^{(i)}-B^{(1)})^{-1}(A^{(i)})^H,
\end{align*}
respectively.
\item[2.]
We notice that two sequences $\{Q^{(k)}\}$ and $\{B^{(k)}\}$ can be respectively written as
\[
Q^{(k)}=\mathcal{F}_{+}^{(2)}(Q^{(k-1)})\,\,\mbox{and}\,\, B^{(k)}=\mathcal{G}_{+}^{(2)}(B^{(k-1)})=f(A\mathcal{G}_{+}(Q-B^{(k-1)})^{-1}A^H),
\]
with the initial matrices~\eqref{initial}.
And the sequences $Q^{(k)}$ and $B^{(k)}$ can be respectively written as
\[
Q^{(k)}=\mathcal{F}_{-}^{(2)}(Q^{(k-1)})\,\,\mbox{and}\,\, B^{(k)}=\mathcal{G}_{-}^{(2)}(B^{(k-1)})=-f(A\mathcal{G}_{-}(Q-B^{(k-1)})^{-1}A^H),
\]
with the initial matrices~\eqref{initial1}.
\end{itemize}
\end{Remark}

In order to study the Eq.~\eqref{eq:NMEP} and Eq.~\eqref{eq:NMEM}  explicitly and conveniently, the rest of this section is divided into two parts to investigate the existence of the positive definite solutions of those two equations.
\subsection{The solvability of Eq.~\eqref{eq:NMEP}}\label{2.1}
To study the existence of the positive definite solutions of the Eq.~\eqref{eq:NMEP},  we need
some conditions on matrices $A$, $B$ and operator $f$. The following lemma provides a mild condition for the existence of the maximal positive definite solutions of the Eq.~\eqref{eq:NMEP}. Note that the sequence of matrices $\{\mathbb{X}^{(k)}\}$ is generated by the iterations~\eqref{signal} with initial matrices $\mathbb{X}^{(1)}$ in \eqref{initial} in this subsection.
\begin{Lemma}\label{Lem1}
For the nonlinear matrix equation~\eqref{eq:NMEP} , suppose that the following condition holds,
    \begin{align}\label{cond1}
     S\neq \phi,\quad \mbox{where }S:=\{X_S>0;X_S\leq\mathcal{F}_{+}(X_S)\}.
    \end{align}
 Then, $\mathbb{X}^{(k)}$ is well-defined for any integer $k\geq 1$. Furthermore, we have the following properties,
 \begin{itemize}
 \item[(1)]
 The condition~\eqref{cond1} implies that $X_S$ is a lower bound of $\{Q^{(k)}\}$ and is an upper bound of $\{B^{(k)}\}$. Furthermore, we have
\begin{align}\label{result1}
 Q\geq Q^{(k)}\geq Q^{(k+1)}\geq X_S>B^{(k+1)}\geq B^{(k)}\geq 0,
\end{align}
 where $X_S$ is any positive definite matrix in $S$.
\item[(2)]$Q^{(k)}\geq \mathcal{F}_+(Q^{(k)})$.
%
\end{itemize}
\end{Lemma}
\begin{proof}
Since part~(1) implies that $Q^{(i)}>B^{(j)}$ with any positive integers $i$ and $j$.
It shall be sufficient to proof the part~(1) and part~(2).
\par\noindent
\begin{itemize}
\item[1.] For part (1), $Q\geq X_S$ is evident. Otherwise, we will prove \eqref{result1} by induction. For $i = 1$ let $\Psi:=\bb f(X_S) & A \\ A^H & Q\eb$. From Lemma~\ref{Schur} we know that $\Psi>0$ since $Q-A^H f(X_S)^{-1} A\geq X_S>0$. Thus, $f(X_S)>A Q^{-1} A^H$. The result $X_S>B^{(1)}$ follows from the assumption that $f$ preserves positivity. On the other hand, we have
 \begin{align*}
 Q-Q^{(1)}&=A^H f(Q)^{-1} A\geq 0,\\
 Q^{(1)}-X_S&\geq A^H(f(X_S)^{-1}-f(Q)^{-1})A\geq 0,\\
 B^{(2)}-B^{(1)}&=A^{(1)}(Q^{(1)}-B^{(1)})^{-1} (A^{(1)})^H\geq 0,\\
 Q^{(1)}-Q^{(2)}&=(A^{(1)})^H (Q^{(1)}-B^{(1)})^{-1} A^{(1)}\geq 0.
\end{align*}
  Now assume that the statement \eqref{result1} is true for $i = s$. Then we want to prove that it holds in the case of $i = s+1$ as well. Let $\Psi_{s+1}:=\bb X_S-B^{(1)} & A^{(1)} \\ (A^{(1)})^H & Q^{(1)}-B^{(s)}\eb$. From Lemma~\ref{Schur} $\Psi_{s+1}>0$ since $Q^{(1)}-B^{(s)}-(A^{(1)})^H (X_S-B^{(1)})^{-1} A^{(1)}\geq X_S-B^{(s)}>0$. Thus, $X_S>B^{(1)}+A^{(1)}(Q^{(1)}-B^{(s)})^{-1}(A^{(1)})^H=B^{(s+1)}$. Note that
 \begin{align*}
 Q-Q^{(s+1)}&=Q-Q^{(s)}+(A^{(s)})^H (Q^{(1)}-B^{(s)})^{-1} A^{(s)}\geq 0,\\
 Q^{(s+1)}-X_S&\geq (A^{(s)})^H( (X_S-B^{(s)})^{-1}-(Q^{(1)}-B^{(s)})^{-1} )A^{(s)}\geq 0,\\
 B^{(s+2)}-B^{(s+1)}&=A^{(1)} ( (Q^{(1)}-B^{(s+1)})^{-1}-(Q^{(1)}-B^{(s)})^{-1} ) (A^{(1)})^H\geq 0,\\
 Q^{(s+1)}-Q^{(s+2)}&=(A^{(s+1)})^H (Q^{(1)}-B^{(s+1)})^{-1} A^{(s+1)}\geq 0.
\end{align*}
The induction process is now finished and thus the result of part (1) is followed.

\item[2.]
     Applying the matrix operator $\mathcal{F}_+$ with $2(k-1)$ times to both side
      $$Q^{(1)}=\mathcal{F}_+(Q)\geq \mathcal{F}_+(Q^{(1)}),$$
      we get $Q^{(k)}\geq \mathcal{F}_+(Q^{(k)})$ for any integer $k\geq 1$.
\end{itemize}
\end{proof}
For the dual Eq.~\eqref{eq:dualNME} with plus sign,
we show that the condition~\eqref{cond1} can be rewritten in an equivalent condition.
\begin{Proposition}\label{Pro}
Let $A$ be a nonsingular matrix. $X_S\in S$ in the condition \eqref{cond1} is equivalent to
    \begin{align*}
    Q-X_S\in\widetilde{S}:=\{\widetilde{X}_S>0;\widetilde{X}_0\leq\mathcal{G}_{+}(\widetilde{X}_S)\}.
    \end{align*}
     In other words, $S\neq\phi$ is equivalent to $\widetilde{S}\neq\phi$.
\end{Proposition}
\begin{proof}
Note that $Q-X_s\geq A^H f(X_s)^{-1}A >0$ since $A$ is nonsingular.
Let $\Psi:=\bb f(X_S) & A \\ A^H & Q-X_S \eb$. Then, $\Psi\geq 0$ if $S\neq \phi$. It is straightforward to show that
$f(X_S)\geq A (Q-X_S)^{-1}A^H$ or $Q-X_S\in\widetilde{S}:=\{\widetilde{X}_S>0;\widetilde{X}_S\leq\mathcal{G}_{+}(\widetilde{X}_S)\}$.
\end{proof}
 In order to perform the main result we also need the following lemma,
\begin{Lemma}\label{coroM1}
Let $A$ be a nonsingular matrix.
Consider the dual equation~\eqref{eq:dualNME} with plus sign
\[
\widetilde{X}=\mathcal{G}_+(\widetilde{X})=Q-f(A)f(\widetilde{X})^{-1}f(A)^{H}.
\]
Suppose that the condition~\eqref{cond1} holds. Then, $\{\widetilde{\mathbb{X}}^{(k)}\}$ generated by the iterations~\eqref{signal} with initial matrices $\widetilde{\mathbb{X}}^{(1)}=(A^H f(Q)^{-1}f(A)^H,A^H  f(Q)^{-1}A,Q-f(A) f(Q)^{-1}f(A)^H)$ is well-defined and
\begin{align}\label{dual}
 A^{(k)}=(\widetilde{A}^{(k)})^H,\,\widetilde{Q}^{(k)}+B^{(k)}=\widetilde{B}^{(k)}+Q^{(k)}=Q,
\end{align}
for each positive integer $k$, where $(A^{(k)},B^{(k)},Q^{(k)})$ is defined in Lemma~\ref{Lem1}.
\end{Lemma}

\begin{proof}
 We will prove \eqref{dual} by induction. For $k=1$ is trivial that \eqref{dual} holds. Now assume that the statement is true for a positive integer $k=s-1$. Together with Proposition~\ref{trans} we have
\begin{align*}
(\widetilde{A}^{(s)})^H&=A^{(s-1)}(Q^{(s-1)}-B^{(1)})^{-1}A^{(1)}={A}^{(s)},\\
B^{(s)}+\widetilde{Q}^{(s)}&=B^{(s)}+(Q-B^{(s-1)})+A^{(s-1)}(Q^{(1)}-B^{(s-1)})^{-1}(A^{(s-1)})^H=Q,\\
Q^{(s)}+\widetilde{B}^{(s)}&=Q^{(s)}+(Q-Q^{(1)})+(A^{(1)})^H(Q^{(s-1)}-B^{(1)})^{-1}A^{(1)}=Q.
\end{align*}
    This shows that~\eqref{dual} is also true for integers $k=s$. By induction principle, \eqref{dual} is true for all positive integers $k$.
\end{proof}
We can now propose our main result for Eq.~\eqref{eq:NMEP} in this subsection.
\begin{Theorem}\label{main1}
Suppose that the operator $f$ is a continuous map from $\mathbb{P}^{n\times n}$ into itself and the assumption~\eqref{cond1} is satisfied. Then the following statements hold,
\begin{itemize}
\item[(1)]$Q^{(\infty)}:=\lim\limits_{\ell\rightarrow\infty}Q^{(\ell)}$ is the maximal positive definite solution of Eq.~\eqref{eq:NMEP}.
\item[(2)]$Q-B^{(\infty)}:=Q-\lim\limits_{\ell\rightarrow\infty}B^{(\ell)}$ is the maximal positive definite solution of dual Eq.~\eqref{eq:dualNME} with plus sign if $A$ is nonsingular.
 \item[(3)]$B^{(\infty)}$ is the minimal positive definite solution of Eq.~\eqref{eq:NMEP} if $A$ is nonsingular.
\end{itemize}
\end{Theorem}
\begin{proof}
\par\noindent
\begin{itemize}
\item[1.]By taking limit as $k\rightarrow\infty$ on both side of part (2) of Lemma~\ref{Lem1} we obtain $\mathcal{F}^{(2)}_+(Q^{(\infty)})=Q^{(\infty)}\geq \mathcal{F}_+(Q^{(\infty)})\geq \mathcal{F}^{(2)}_+(Q^{(\infty)})$. That is,
    $Q^{(\infty)}=\mathcal{F}_+(Q^{(\infty)})$.
\item[2.]
From Lemma~\ref{Lem1} and Proposition~\ref{Pro} it follows that the dual equation~\eqref{eq:dualNME} with plus sign has maximal Hermitian positive definite solution $\widetilde{Q}^{(\infty)}$. Next,
the result immediately follows from the foregoing Lemma~\ref{coroM1}.
\item[3.]When $A$ is nonsingular,  it is easy to check that $\widetilde{X}=\mathcal{G}_+(\widetilde{X})$ is equivalent to $X=\mathcal{F}_+(X)$, where $\widetilde{X}:=Q-X$. From part~(2) it follows that
    \[
Q-\lim\limits_{\ell\rightarrow\infty}B^{(\ell)}=\widetilde{X}_\infty\geq \widetilde{X}=Q-X.
    \]
This completes the proof of part~(3).
\end{itemize}
\end{proof}


\subsection{The solvability of Eq.~\eqref{eq:NMEM}}
As aforementioned above, Eq.~\eqref{eq:NMEM} can be also transformed into the standard nonlinear matrix equation~\eqref{NMEP} with coefficient matrices~\eqref{initial1}. Let the sequence of matrices $\{\mathbb{X}^{(k)}\}$ be generated by the iterations~\eqref{signal} with initial matrices $\mathbb{X}^{(1)}$ in \eqref{initial1}. Analogously to the foregoing result we can verify the following consequences.
\begin{Lemma}\label{Lem2}
\par\noindent
\begin{itemize}
 \item[a.]
 Let the set of Hermitian matrix solutions of the nonlinear matrix equation~\eqref{eq:NMEM} be
    \begin{align}\label{cond2}
     K:=\{X_K\in\mathbb{H}^{n\times n};X_K=\mathcal{F}_{-}(X_K)\}.
    \end{align}
Then, $K\cap \mathbb{P}^{n\times n}$ is nonempty if $f$ is a continuous map from $\mathbb{P}^{n\times n}$ into itself. That is,  Eq.~\eqref{eq:NMEM} always has a positive definite solution $X$. Furthermore, $K$ admits a negative definite solution if $A$ is nonsingular.
\item[b.]
 For any integer $k\geq 1$, the sequence $\{\mathbb{X}^{(k)}\}$ is well-defined. Furthermore, we have two following properties,
 \begin{itemize}
 \item[(1)]
 $X_K$ is a lower bound of $\{Q^{(k)}\}$ and is an upper bound of $\{B^{(k)}\}$. Furthermore,
 $$ B^{(k)}\leq B^{(k+1)}\leq 0< Q\leq X_K \leq Q^{(k+1)}\leq Q^{(k)},$$
where $X_K\in K\cap \mathbb{P}^{n\times n}$.
 \item[(2)]
Suppose that $A$ is nonsingular. Let $f$ be a continuous map from $\mathbb{P}^{n\times n}$ into itself, then  $Q^{(\infty)}:=\lim\limits_{\ell\rightarrow\infty}Q^{(\ell)}$ is a unique positive definite solution of Eq.~\eqref{eq:NMEM}.
\end{itemize}
\end{itemize}
\end{Lemma}
\begin{proof}
 \par\noindent
 \begin{itemize}
  \item[a.]
First given a positive definite $X$ such that $Q \leq X \leq Q^{(1)}$.
It is clear that $Q\leq \mathcal{F}_-(X)=Q+A^H f(X)^{-1} A \leq Q+A^H f(Q)^{-1} A=Q^{(1)}$. The existence of the positive solution of Eq.~\eqref{eq:NMEM} follows immediately Schauder's fixed point theorem under the assumption that $f$ is continuous on $\mathbb{P}^{n\times n}$. Let $A$ be a nonsingular matrix. Then, $\widetilde{X}:=Q-X=-A^H f(X)^{-1} A$ is nonsingular and Eq.~\eqref{eq:NMEM} is equivalent to $\widetilde{X}=\mathcal{G}_-(\widetilde{X})$. Thus, there exists a negative definite solution $Y$ of Eq.~\eqref{eq:NMEM} such that $-A^H f(Q)^{-1} A\leq Y < 0$.
 \item[b.]
Since part~(1) implies that $Q^{(i)}>0\geq B^{(j)}$ with any positive integers $i$ and $j$.
It shall be sufficient to proof the part~(1)--part~(2).
\begin{itemize}
\item[1.] For part (1), $B^{(1)}\leq 0$ is evident. 
 Otherwise, we will prove this by induction. It follows that
 \begin{align*}
 Q-Q^{(1)}&=-A^H f(Q)^{-1} A\leq 0,\\
 Q^{(1)}-X_K&=A^H(f(Q)^{-1}-f(X_K)^{-1})A\geq 0,\\
 B^{(2)}-B^{(1)}&=A^{(1)}(Q^{(1)}-B^{(1)})^{-1} (A^{(1)})^H\geq 0,\\
 Q^{(1)}-Q^{(2)}&=(A^{(1)})^H (Q^{(1)}-B^{(1)})^{-1} A^{(1)}\geq0.
\end{align*}
  Now assume that the inequality in part~(1) is true for $i = s$, then we want to prove that it also holds for $i = s+1$. Note that
 \begin{align*}
 B^{(s+1)}&=-f(A)(f(Q)+A(Q-B^{(s)})^{-1}A^H)^{-1} f(A)^H\leq 0,\\
 Q^{(s+1)}-X_K&= (A^{(s)})^H( (X_K-B^{(s)})^{-1}-(Q^{(1)}-B^{(s)})^{-1} )A^{(s)}\geq 0,\\
 B^{(s+2)}-B^{(s+1)}&=A^{(1)} ( (Q^{(1)}-B^{(s+1)})^{-1}-(Q^{(1)}-B^{(s)})^{-1} ) (A^{(1)})^H\geq 0,\\
 Q^{(s+1)}-Q^{(s+2)}&=(A^{(s+1)})^H (Q^{(1)}-B^{(s+1)})^{-1} A^{(s+1)}\geq 0.
\end{align*}
So part~(1) also holds for $i = s + 1$, which we have shown.

\item[2.]
First, the matrix equation $X=\mathcal{F}_-^{(2)}(X)$ is equivalent to the discrete algebraic Riccati equation
     \[
     X=Q+\widehat{A}^H X \widehat{A}-\widehat{A}^H X (X-B^{(1)})^{-1}X\widehat{A},
     \]
    where $\widehat{A}=f(A)^{-H}A$ is nonsingular. Let $Q=Q_1^HQ_1$ for a positive definite matrix $Q_1$. It is clear that $(\widehat{A},I_n)$ is stabilizable and $(\widehat{A},Q_1)$ is detectable.
    From part~(1) and \cite{Bini2012} or \cite[Theorem 5.6]{Sayed01}, $Q^{(\infty)}$ is the unique positive definite solution of the equation $X=\mathcal{F}_-^{(2)}(X)$. Thus, $Q^{(\infty)}=\mathcal{F}_-(Q^{(\infty)})>0$ is the unique positive definite solution of the equation of Eq.~\eqref{eq:NMEM} since
    Eq.~\eqref{eq:NMEM} always has a positive definite solution.
\end{itemize}
\end{itemize}
\end{proof}
The proofs of the following results follow in a similar manner as the proofs of Lemma~\ref{coroM1}, and Theorem~\ref{main1}. We omit it here.

\begin{Lemma}\label{coroM2}
Consider the dual equation~\eqref{eq:dualNME} with minor sign
\[
\widetilde{X}=\mathcal{G}_-(\widetilde{X})=Q+f(A)f(\widetilde{X})^{-1}f(A)^{H}.
\]
Then, the sequence $\{\widetilde{\mathbb{X}}^{(k)}\}$ is well-defined by the iterations~\eqref{signal} with initial matrices
$$\widetilde{\mathbb{X}}^{(1)}=(A^H f(Q)^{-1}f(A)^H,-A^H f(Q)^{-1}A,Q+f(A) f(Q)^{-1}f(A)^H)$$
 and
\begin{align*}
 A^{(k)}=(\widetilde{A}^{(k)})^H,\,\widetilde{Q}^{(k)}+B^{(k)}=\widetilde{B}^{(k)}+Q^{(k)}=Q,
\end{align*}
for each positive integer $k$, where $(A^{(k)},B^{(k)},Q^{(k)})$ is defined in Lemma~\ref{Lem2}.
\end{Lemma}

\begin{Theorem}\label{main2}
Suppose that the operator $f$ is a continuous map from $\mathbb{P}^{n\times n}$ into itself and $A$ is nonsingular. Then, the following statements hold.
\begin{itemize}
\item[(1)]$Q^{(\infty)}:=\lim\limits_{\ell\rightarrow\infty}Q^{(\ell)}$ is the maximal positive definite solution of Eq.~\eqref{eq:NMEM}.
\item[(2)]$Q-B^{(\infty)}$ is the maximal positive definite solution of dual Eq.~\eqref{eq:dualNME} with minor sign.
 \item[(3)]$B^{(\infty)}:=\lim\limits_{\ell\rightarrow\infty}B^{(\ell)}$ is the minimal negative definite solution of Eq.~\eqref{eq:NMEP}.
\end{itemize}
\end{Theorem}
\begin{Remark}
\par\noindent
\begin{itemize}
\item[1.]
Let $f$ be a matrix operator with only the assumption (a4). The existence of solutions of nonlinear matrix equations of the kind $X\pm A^H f(X) A=Q$ has been studied extensively. It is worthwhile to mention that El-Sayed et al. provide the necessary and sufficient conditions~\cite[Theorem 3.1]{Sayed01} of existence of a positive definite solution of a generalization of Eq.~\eqref{eq:NMEP}. Moreover, some sufficient conditions for the existence of a positive semidefinite solution of Eq.~\eqref{eq:NMEM} are obtained in \cite[Lemma 2.2]{Ran200215}.
\item[2.]
In the part~(3) of Theorem~\ref{main1} and \ref{main2}, it is easy to show that $\mbox{rank }(B^{(k)})=\mbox{rank }(A)=\mbox{rank }(A^{(k)})$ for any positive integer $k$ from the part~(2) of Remark~\ref{Rmk}. Thus $B^{(\infty)}$ is not the solution (positive or negative definite) of Eqs.~\eqref{eq:NME} if $A$ is singular.
\end{itemize}
\end{Remark}
\section{The convergence analysis of iteration~\eqref{signal}}
In this section we will study the numerical behavior of iteration~\eqref{signal} with initial matrices~\eqref{initial2}. For the sake of simplicity we denote the maximal positive definite solutions of Eqs.~\eqref{eq:NME} and dual Eqs.~\eqref{eq:dualNME} by $X_M$ and $Q-Y_M$, respectively. As mentioned before, if Eq.~\eqref{eq:NMEP} (or Eq.~\eqref{eq:NMEM}) have a symmetric positive definite solution, then $X_M$ exists and the sequence $\{Q^{(k)}\}$ converges to $X_M$ of Eq.~\eqref{eq:NMEP} (or Eq.~\eqref{eq:NMEM} if $A$ is nonsingular). As a summary of previous section, the following recursive algorithm is presented to compute $X_M$ and $Y_M$ under some mild conditions.
 \begin{Algorithm}\label{fix1}
{\emph{(Fixed point iteration for solving Eqs.~\eqref{eq:NME})}}
\begin{enumerate}
\item $({A}^{(1)},{B}^{(1)},{Q}^{(1)})$ \mbox{ is given in \eqref{initial2}}.
\item \emph{For} $k=1,2,\ldots,$ \emph{compute}

\begin{align*}
A^{(k+1)}&:=A^{(1)}(Q^{(1)}-B^{(k)})^{-1}A^{(k)},\\
B^{(k+1)}&:=B^{(1)}+A^{(1)}(Q^{(1)}-B^{(k)})^{-1}(A^{(1)})^H,\\
Q^{(k+1)}&:=Q^{(k)}-(A^{(k)})^H (Q^{(1)}-B^{(k)})^{-1} A^{(k)},
\end{align*}
 until convergence.
\item $X_M=Q^{(\infty)}$ and $Y_M=B^{(\infty)}$.
\end{enumerate}
\end{Algorithm}
 Naturally, we are interested in the rate of convergence and the error estimate formula on this iterative method. To begin with, suppose that $f$ is a continuous operator, the hypotheses~\eqref{cond1}  corresponding to Eq.~\eqref{eq:NMEP} holds and $A$ is nonsingular in Eq.~\eqref{eq:NMEM}
 through this section. Let the sequence $\{\mathbb{X}^{(k)}\}$ be generated by Algorithm~\ref{fix1}. The following results play an important role in this section.
\begin{Lemma}\label{Lem3}
 Let $T_k:=(X_M-B^{(k)})^{-1}A^{(k)}$ and $S_k:=A^{(k)}(Q^{(k)}-Y_M)^{-1}$ for each positive integer $k$. Then, the following results hold,
\begin{itemize}
\item[(1)] 
\begin{subequations}
\begin{align}
\Delta^{(k)}&:=(X_M-B^{(k)})^{-1}= \sum\limits_{i=0}^{k-1} (T_1)^i \Delta^{(1)} (T_1^H)^i,\label{sum1}\\
\widetilde{\Delta}^{(k)}&:=(Q^{(k)}-Y_M)^{-1}= \sum\limits_{i=0}^{k-1} (S_1^H)^i \widetilde{\Delta}^{(1)} (T_1)^i,\label{sum2}
\end{align}
\end{subequations}
and equations \eqref{sum1} and \eqref{sum2} can be rewritten as
\begin{subequations}
\begin{align}
\Delta^{(k)}-T_1 \Delta^{(k)}  T_1^H&=\Delta^{(1)} -T_1^k\Delta^{(1)} (T_1^H)^{k},\label{ste1}\\
\widetilde{\Delta}^{(k)}-S_1 \widetilde{\Delta}^{(k)}  S_1^H&=\widetilde{\Delta}^{(1)} -S_1^k\widetilde{\Delta}^{(1)} (S_1^H)^{k},\label{ste2}
\end{align}
\end{subequations}
respectively.
\item[(2)]
\begin{align*}
T_k=T_1^k,&\,S_k=S_1^k,\\
Q^{(k)}-X_M&=T_k^H (X_M-B^{(k)}) T_k,\\
Y_M-B^{(k)}&=S_k (Q^{(k)}-Y_M) S_k^H.
\end{align*}
\item[(3)]
$T_1S_1^H=\Delta^{(1)}(Y_M-B^{(1)}),\,S_1^H T_1=\widetilde{\Delta}^{(1)}(X_M-Q^{(1)}).$ Furthermore, all eigenvalues of $T_1S_1^H$ are real and nonnegative.
\end{itemize}
\end{Lemma}
\begin{proof}
\par\noindent
\begin{itemize}
\item[1.]
First note that $X_M-B^{(k)}=(X_M-B^{(1)})-A^{(1)}(Q^{(1))}-B^{(k-1)})^{-1}(A^{(1)})^H$ and $X_M-B^{(k-1)}= Q^{(1)}-B^{(k-1)}-(A^{(1)})^H(X_M-B^{(1)})^{-1}A^{(1)}$.
Applying SMWF  we conclude that
\begin{align*}
\Delta^{(k)}= \Delta^{(1)}+T_1\Delta^{(k-1)} T_1^H.
\end{align*}
Equalities~\eqref{sum1} and \eqref{ste1} immediately follow by induction. the proof of two equalities~\eqref{sum2} and \eqref{ste2} is analogous to the proof above.
\item[2.]
Let us first define two matrices $(\mathcal{M}^{(k)},\mathcal{L}^{(k)})$ with a positive integer $k$,
\begin{equation*}
   \mathcal{M}^{(k)}:=
   \left [
   \begin{array}{rc} A^{(k)} & 0 \\ Q^{(k)} & -I_n
   \end{array} \right ] , \quad
   \mathcal{L}^{(k)}:=\left [ \begin{array}{lc} -B^{(k)} & I_n \\ (A^{(k)})^H & 0 \end{array}
   \right ].
\end{equation*}
We can easily prove the following identities,
\begin{subequations}\label{ml}
\begin{align}
&\mathcal{M}^{(k)} \left [ \begin{array}{c} I_n \\ X_M \end{array}\right ] = \mathcal{L}^{(k)}
\left [ \begin{array}{c} I_n \\ {X_M} \end{array}\right ]T_k,\\
&\mathcal{M}^{(k)} \left [ \begin{array}{c} I_n \\ Y_M \end{array}\right ]S_k^H = \mathcal{L}^{(k)}
\left [ \begin{array}{c} I_n \\ Y_M \end{array}\right ].
\end{align}
\end{subequations}
Also, assume further that
\begin{align*}
\mathcal{M}^{(k)}_{\star} :&= \left [
\begin{array}{rc} A^{(1)} (Q^{(1)} -B^{(k)} )^{-1} & 0
\\ -(A^{(1)})^H (Q^{(1)} -B^{(k)})^{-1} & I_n \end{array} \right ],\,\\
\mathcal{L}^{(k)}_{\star} :&=
\left [
\begin{array}{cr}
I_n & -A^{(1)}(Q^{(1)} -B^{(k)})^{-1} \\ 0 & (A^{(1)})^H (Q^{(1)} -B^{(k)})^{-1} \end{array}
\right ].
\end{align*}
By direct computation we have $\mathcal{M}^{(k)}_{\star}\mathcal{L}^{(k)}=\mathcal{L}^{(k)}_{\star}\mathcal{M}^{(1)}$, and
\[
\mathcal{M}^{(k+1)}=\mathcal{M}^{(k)}_{\star}\mathcal{M}^{(k)},\,\mathcal{L}^{(k+1)}=\mathcal{L}^{(k)}_{\star}\mathcal{L}^{(1)}.
\]
Thus, $T_{k+1}=T_kT_1$ and $S_{k+1}=S_1S_k$ by direct inspection. The equalities $T_{k}=T_1^k$ and $S_{k}=S_1^k$ are proved by induction. Finally,
comparing both sides of \eqref{ml} yields
\begin{subequations}\label{conv1}
\begin{align}
Q^{(k)}-X_M &= (T_1^k)^H (X_M-B^{(k)}) T_1^k,\\
Y_M-B^{(k)}&= S_1^k (Q^{(k)}-Y_M) (S_1^k)^H.
\end{align}
\end{subequations}
The remaining part of part~(2) immediately follows.
\item[3.]
The first part follows from direct computations and it is omitted.
The second part is obvious from the fact that all eigenvalues of the product of two positive semidefinite matrices are real and nonnegative.

\end{itemize}
\end{proof}
 The following proposition concerning perturbation theory for the operator $\mathcal{F}_+$ is also needed in the
proof of the main result.
\begin{Proposition}\label{comparision}
We consider the nonlinear matrix equation
\begin{align*}
X=\mathcal{F}_+^\epsilon(X):=Q_\epsilon-A^{H}f(X)^{-1}A
\end{align*}
with $Q_\epsilon>Q$. Then, for any integer $k>0$,
\[
Q^{(k)}_\epsilon-Q^{(k)}\geq Q_\epsilon-Q,\quad B_\epsilon^{(k)}\leq B^{(k)}.
\]
\end{Proposition}
\begin{proof}
It is clear that $\mathcal{F}^\epsilon_+(Q^{(k)}_\epsilon)>\mathcal{F}_+(Q^{(k)})$ and thus $Q^{(k+1)}_\epsilon-Q^{(k+1)}=(\mathcal{F}_{+}^\epsilon)^{(2)}(Q^{(k)}_\epsilon)-\mathcal{F}^{(2)}_{+}(Q^{(k)})
=(Q_\epsilon-Q)+A^H(f(\mathcal{F}_+(Q^{(k)}))^{-1}-f(\mathcal{F}^\epsilon_+(Q^{(k)}_\epsilon))^{-1})A\geq Q_\epsilon-Q$.
On the other hand, $ B_\epsilon^{(k+1)}=f(A)(f(Q_\epsilon)-A(Q_\epsilon-B_\epsilon^{(s)})^{-1}A^H)^{-1} f(A)^H\leq f(A)(f(Q)-A(Q-B^{(s)})^{-1}A^H)^{-1} f(A)^H=B^{(k+1)}$.
\end{proof}
The main theorem of this section is stated below.
\begin{Theorem}
For nonlinear matrix equations~\eqref{eq:NME},
we have $\sigma(T_1)=\sigma(S_1)$,
 $\rho(T_1)\leq 1$ and $\rho(T_1S_1^H)\leq 1$. Furthermore,
\begin{itemize}
\item[(1)]
$X_M>Y_M$ if and only if $\rho(T_1)<1$ if and only if $\rho(T_1S_1^H)< 1$.
\item[(2)] For Eq.~\eqref{eq:NMEM}, $X_M-Y_M$ is always a positive definite matrix. In other words,
$\rho(T_1)$ is forever less than one .
\item[(3)]For Eq.~\eqref{eq:NMEP},  $X_M-Y_M$ is singular if and only if $\rho(T_1)= 1$ if and only if $1\in\sigma(T_1S_1^H)$. Moreover, the dimension of the null space of $X_M-Y_M$ is equal to the algebraic multiplicity of the one eigenvalue of $T_1S_1^H$.
\item[(4)]
All sequences generated by Algorithm~\ref{fix1} are well-defined. In addition, the convergence speed is R-linearly if $\rho(T_1)<1$ and the convergence rate can be shown
\begin{subequations}
\begin{align*}
&\limsup\limits_{k\rightarrow\infty}\sqrt[k]{\| {A}^{(k)}\|}\leq \rho(T_1),\\
&\limsup\limits_{k\rightarrow\infty}\sqrt[k]{\| {Q}^{(k)}-X_M\|}\leq \rho(T_1)^2,\\
&\limsup\limits_{k\rightarrow\infty}\sqrt[k]{\| Q-{B}^{(k)}-Y_M\|}\leq \rho(T_1)^2.
\end{align*}
\end{subequations}
\end{itemize}
\end{Theorem}
\begin{proof}
Since $X_M=\lim\limits_{k\rightarrow\infty}Q^{(k)}\geq\lim\limits_{k\rightarrow\infty}B^{(k)}=Y_M$,
it suffices to show the result of part~(1), part~(2) and part~(3), and part~(4) immediately
follows.
\begin{itemize}
\item[1.]
Suppose that $X_M$ is greater than $Y_M$. Since each term of right hand side of \eqref{sum1} is positive definite, we have $T_1^i (\Delta^{(1)})^{1/2}\rightarrow 0$ as $i\rightarrow\infty$. Thus, $\rho(T_1)<1$.
 Similarly one can prove that $\rho(S_1)<1$. Conversely, the condition $\rho(T_1)<1$
guarantees the existence of a unique positive definite solution $X=\Delta^{(\infty)}$ of the following matrix equation
\[
X-T_1 X  T_1^H=\Delta^{(1)},
\]
which is the Stein matrix equation~\eqref{ste1} when $k\rightarrow \infty$. Therefore $X_M>Y_M$.
In addition, it can easily be checked that
\begin{align*}
p(\lambda):=\det(\mathcal{M}^{(1)}-\lambda\mathcal{L}^{(1)})&=\det(\overline{\mathcal{L}^{(1)}}-{\lambda}\overline{\mathcal{M}^{(1)}})=\lambda^{2n}\overline{p(1/\overline{\lambda})},
\end{align*}
 where $\lambda\in\mathbb{C},$ $\lambda\neq 0$. Namely, $p(\lambda)$ is a conjugate reciprocal polynomial~\cite{roman2005field}. It implies that the conjugate-reciprocity property, i.e., $1/\bar{\lambda}\in \sigma(\mathcal{M}^{(1)}-\lambda\mathcal{L}^{(1)})$ if $\lambda\in\sigma(\mathcal{M}^{(1)}-\lambda\mathcal{L}^{(1)})$ $(1/0:=\infty)$. Moreover, the algebraic multiplicity of $\lambda_0\in\sigma(\mathcal{M}^{(1)}-\lambda\mathcal{L}^{(1)})$ is equal to the algebraic multiplicity of $1/\overline{\lambda}_0$. Together with \eqref{ml} the spectral of $T_1$ is coincident with the spectral of $S_1^H$.
Finally, together with the part (3) of Lemma~\ref{Lem3}, the last necessary and sufficient condition follows
 from Lemma~\ref{Schur}.
\item[2.] From the part (1) of Lemma~\ref{Lem2} it implies that
\[
Y_M\leq 0 <Q\leq X_M,
\]
the result immediately follows from the foregoing conclusion.
\item[3.]
In the case of $X_M-Y_M\geq 0$, then $Q_\epsilon^{(\infty)}-B_\epsilon^{(\infty)}\geq Q_\epsilon-Q>0$ for arbitrary $\epsilon>0$ from Proposition~\ref{comparision}. Based on the continuity argument, $\rho(T_1)=\lim\limits_{\epsilon\rightarrow 0^+}T^\epsilon_1\leq 1$ if and only if $X_M-Y_M\geq 0$. And, $\rho(T_1)=1$ is equivalent to
$X_M-Y_M$ is singular. Finally, the remaining part now follows from the foregoing result that $I_n-T_1S_1^H=(X_M-B^{(1)})^{-1}(X_M-Y_M).$ We remark that the equality $\sigma(T_1)=\sigma(S_1)$ is still correct in this situation.
\item[4.] Combining equalities~\eqref{conv1} with previous Lemmas~\ref{Lem1} and \ref{Lem2} gives this result.
\end{itemize}
\end{proof}

\begin{Remark}
In the paper we study the existence of the maximal positive definite solution
of Eq.~\eqref{eq:NMEP} or Eq.~\eqref{eq:NMEM}  by means of the existence of the maximal positive definite solution of the standard nonlinear matrix equation~\eqref{NMEP}.
As is well-known, the existence of a symmetric
positive definite solution and a maximal symmetric positive definite
solution of \eqref{NMEP} has been established in
\cite{Engweradaran93}. The result in \cite{Engweradaran93} is obtained by utilizing an analytic factorization approach. We state with a review of this
result as following:
\begin{Theorem}\cite{Engweradaran93}
Let $\psi(\lambda)$ be a rational matrix-valued function defined by
\begin{equation*}
\psi(\lambda) = Q^{(1)}-B^{(1)} + \lambda A^{(1)} + \lambda^{-1} (A^{(1)})^H.
\end{equation*}
Then, the standard NME~\eqref{NMEP} has a symmetric positive definite
    solution if and only if the following two assumptions hold,
 \begin{itemize}
 \item[(F1)]$\psi ( \lambda )$ is regular, i.e., $\det\psi (\lambda)\neq 0$ for some $\lambda\in\mathbb{C}$.
 \item[(F2)]$\psi ( \lambda )$ is nonnegative on unit circle, i.e.,
  $\psi ( \lambda ) \geq 0$ for all $| \lambda | = 1$.
\end{itemize}
In that case $\psi(\lambda)$ factors as the well-known operator-valued Fej$\acute{e}$r-Riesz factorization:
    \[
\psi(\lambda)=(C_0^\ast+\lambda^{-1}C_1^\ast)(C_0+\lambda C_1)
    \]
with $\det(C_0)\neq 0$, then $X=B^{(1)}+C_0^\ast C_0$ is a solution of
Eq.~\eqref{NMEP}. Every positive definite solution is obtained in
this way. Moreover, for
the maximal solution $X_M$, we have
$\rho((X_M-B^{(1)})^{-1}A^{(1)})\le 1$ and for any other symmetric
positive definite solution $X$, we have
$\rho((X_M-B^{(1)})^{-1}A^{(1)})> 1$.
\end{Theorem}
Let $f$ be the identity operator in Eq.~\eqref{eq:NMEP}, we contribute a different approach to the necessary and sufficient conditions for the existence of maximal positive definite solution. That is,
the condition~\eqref{cond1} is equivalent to conditions~(F1) and (F2). Also, we investigate the relationship between
the nullity of the matrix $X_M-Y_M$ and the spectral radius of $T_1 $, which is important to clarify the convergence speed of Algorithm~\ref{fix1}.
As compared to earlier work on this topic, the results here are obtained with only
using elementary matrix theory, and the analysis here is much simpler.
\end{Remark}



\section{Two Accelerative iterations}
According to the foregoing discussions, we know that solving the maximal positive definite solutions $X_M$ of Eqs.~\eqref{eq:NME}
 is equivalent to finding the maximal positive definite solution $X_M$ of the standard nonlinear matrix equation Eq.~\eqref{NMEP}. The standard approach for solving the Eq.~\eqref{NMEP} is to compute its generalized Lagrangian eigenspaces of a certain matrix pencil~\cite{Engweradaran93}. Otherwise, the fixed-point iteration in Algorithm~\ref{fix1} is a basic and simple method for solving the maximal positive definite solutions of the Eq.~\eqref{NMEP}. However, the convergence speed of all of these methods have been shown to be very slow while $\rho(T_1)$ is very close to 1 or $T_1$ has eigenvalues on the unit circle. When $\rho(T_1)$ is sufficiently
small, a method for choosing the initial guess~$Q^{(1)}$ for fixed-point iterations~\eqref{fixQ} was introduced in \cite{10.2307/4100245} that guarantees a faster convergence rate to the maximal positive definite solutions of Eq.~\eqref{NMEP}. In \cite{Lin06}, a structure-preserving doubling algorithm (SDA) was
proposed for finding the maximal positive definite solution $X_M$ of the Eq.~\eqref{NMEP},
and, it was proven that this algorithm
 converges R-quadratically when all eigenvalues
of $T_1$ lie inside the unit circle. Moreover, the convergence rate is at least R-linear with rate $1/2$ when each iteration $Q^{(k)}-B^{(k)}$ is invertible and $B^{(k)}$ is bounded~\cite{ChiangPHD2009}. Note that other iterative solution processes, by using Newton's iteration or cyclic reduction, have been introduced in
\cite{Guo_01SIMAA,GuoLan_99_MC,Meini_02_MC} and
linear convergence for problems with
semi-simple unimodular eigenvalues has been
observed and proved in \cite{Guo_01SIMAA}.

There are several techniques for convergence acceleration of the sequences produced by fixed point iteration. By further analyzing of the deep structure of the iteration~\eqref{iteration} in the previous discussion,
 we are going to give an theoretical interpretation for this by using the well properties~\eqref{ijk} in this section. We first present a new accelerative iteration that contains the original iteration~\eqref{signal} and a special initial matrices. Second, we propose an iterative method with R-superlinear with order $r$ from a very simple point of view, where $r$ is a given integer greater than $1$.

 Assume that the hypotheses~\eqref{cond1} holds and suppose that $f$ is a continuous operator through this section. Let $\widehat{\mathbb{X}}^{(k)}:=\mathcal{I}_{\mathbb{X}^{(1)}}(\widehat{\mathbb{X}}^{(k-1)},\mathbb{X}^{(\ell)})$ with a prescribed positive integer $\ell$. It is known that $\widehat{\mathbb{X}}^{(k)}={\mathbb{X}}^{(k\ell)}$ by applying the group-like law~\eqref{tran1}. That is, the original nonlinear matrix equation~\eqref{NMEP} becomes the standard nonlinear matrix equation~\eqref{kNME} by applying $\mathcal{F}$ with $\ell$ times. Finally $\widehat{\mathbb{X}}^{(k)}$ can be designed as the following accelerative iteration.
\begin{Algorithm}\label{aa1}
{\emph{(Accelerative iteration 1 for solving Eqs.~\eqref{eq:NME})}}
\begin{enumerate}
\item \emph{Given a prescribed positive integer $\ell$}.
\item $(\widehat{A}^{(1)},\widehat{B}^{(1)},\widehat{Q}^{(1)})=({A}^{(\ell)},{B}^{(\ell)},{Q}^{(\ell)})$ \mbox{ is provided in Algorithm \ref{fix1}}.
\item \emph{For} $k=1,2,\ldots,$ \emph{compute}
\begin{align*}
\widehat{A}^{(k+1)}&:=\widehat{A}^{(1)}(\widehat{Q}^{(1)}-\widehat{B}^{(k)})^{-1}\widehat{A}^{(k)},\\
\widehat{B}^{(k+1)}&:=\widehat{B}^{(1)}+\widehat{A}^{(1)}(\widehat{Q}^{(1)}-\widehat{B}^{(k)})^{-1}(\widehat{A}^{(1)})^H,\\
\widehat{Q}^{(k+1)}&:=\widehat{Q}^{(k)}-(\widehat{A}^{(k)})^H (\widehat{Q}^{(1)}-\widehat{B}^{(k)})^{-1} \widehat{A}^{(k)},
\end{align*}
until convergence.
\end{enumerate}
\end{Algorithm}
Recall that all sequences generated by Algorithm~\ref{aa1} are well-defined, and
 the convergence speed is R-linearly if $\rho(T_1)<1$ and the convergence rate can be shown
\begin{subequations}
\begin{align*}
&\limsup\limits_{k\rightarrow\infty}\sqrt[k]{\| \widehat{A}^{(k)}\|}\leq \rho(T_1)^{\ell},\\
&\limsup\limits_{k\rightarrow\infty}\sqrt[k]{\| \widehat{Q}^{(k)}-X_M\|}\leq \rho(T_1)^{2\ell},\\
&\limsup\limits_{k\rightarrow\infty}\sqrt[k]{\| Q-\widehat{B}^{(k)}-Y_M\|}\leq \rho(T_1)^{2\ell}.
\end{align*}
\end{subequations}
 We remark that the accelerative iteration~\ref{aa1} is basically a fixed-point iteration for solving the maximal positive solution $X_M$ of Eq.~\eqref{NMEP}.

In order to maintain an accelerative iteration that converges R-superlinearly with order $r$ to $X_M$  of Eq.~\eqref{NMEP}, let $\widehat{\mathbb{X}}^{(k)}:=\mathbb{Y}^{(2)}$ for any integer $k>1$ and $\widehat{\mathbb{X}}^{(1)}:=\mathbb{X}^{(1)}$,  where $\mathbb{Y}^{(s)}:=\mathcal{I}_{\widehat{\mathbb{X}}^{(k-1)}}(\mathbb{Y}^{(s-1)},{\mathbb{X}}_{r-1}{(k-1)})$
 and ${\mathbb{X}}_{r-1}{(k)}$ are recursively defined by
\begin{align*}
{\mathbb{X}}_{i}{(k)}&:=\mathbb{Y}_{i}^{(2)}\,\mbox{with}\,\mathbb{Y}_{i}^{(s)}:=\mathcal{I}_{\widehat{\mathbb{X}}^{(k)}}(\mathbb{Y}_{i}^{(s-1)},{\mathbb{X}}_{i-1}{(k)}),\quad 2\leq i\leq r-1.\\
{\mathbb{X}}_{1}{(k)}&:=\widehat{\mathbb{X}}^{(k)},
\end{align*}
with a prescribed positive integer $r>1$. Then, ${\mathbb{X}}_{i}{(k)}=\mathbb{X}^{(ir^{k-1})}$ can be easily verified from Proposition~\ref{trans} for any integer $k\geq 1$ and thus $\widehat{\mathbb{X}}^{(k)}=\mathbb{X}^{(r^{k-2}+(r-1)r^{k-2})}=\mathbb{X}^{(r^{k-1})}$ for any integer $k> 1$.
Summary, $\widehat{\mathbb{X}}^{(k)}$ can be designed as the following recursive algorithm.
\begin{Algorithm}\label{aa2}
{\emph{(Accelerative iteration 2 for solving Eqs.~\eqref{eq:NME})}}
\begin{enumerate}
\item \emph{Given a prescribed positive integer $r>1$}.
\item  $(\widehat{A}^{(1)},\widehat{B}^{(1)},\widehat{Q}^{(1)})$\quad \mbox{is given in \eqref{initial2}}.
\item    \emph{For} $k=1,2,\ldots,$ \emph{compute}
 \begin{align*}
\widehat{A}^{(k+1)}&=\widehat{A}^{(k)}(\widehat{Q}^{(k)}-B_{r-1}{(k)})^{-1}A_{r-1}{(k)},\\
\widehat{B}^{(k+1)}&=\widehat{B}^{(k)}+\widehat{A}^{(k)}(\widehat{Q}^{(k)}-B_{r-1}{(k)})^{-1}(\widehat{A}^{(k)})^H,\\
\widehat{Q}^{(k+1)}&=Q_{r-1}{(k)}-A_{r-1}{(k)}^H (\widehat{Q}^{(k)}-B_{r-1}{(k)})^{-1} A_{r-1}{(k)},
\end{align*}
    until convergence, where $({A}_{r-1}{(k)},{B}_{r-1}{(k)},{Q}_{r-1}{(k)})$ is defined in step 4.
\item
     \emph{For} $i=1,\cdots,r-2$, iterate
     \begin{align*}
A_{i+1}{(k)}&=\widehat{A}^{(k)}(\widehat{Q}^{(k)}-B_{i}{(k)})^{-1}A_{i}{(k)},\\
B_{i+1}{(k)}&=\widehat{B}^{(k)}+\widehat{A}^{(k)}(\widehat{Q}^{(k)}-B_{i}{(k)})^{-1}(\widehat{A}^{(k)})^H,\\
Q_{i+1}{(k)}&=Q_{i}{(k)}-A_{i}{(k)}^H (\widehat{Q}^{(k)}-B_{i}{(k)})^{-1} A_{i}{(k)},
\end{align*}
with $(A_1{(k)},B_1{(k)},Q_1{(k)})=(\widehat{A}^{(k)},\widehat{B}^{(k)},\widehat{Q}^{(k)})$.
 \end{enumerate}
\end{Algorithm}
As we have already discussed that all sequences generated by Algorithm~\ref{aa2} are well-defined, the convergence speed is R-superlinearly with order $r$ if $\rho(T_1)<1$ and
the convergence rate can be shown
\begin{subequations}
\begin{align*}
&\limsup\limits_{k\rightarrow\infty}\sqrt[r^k]{\| \widehat{A}^{(k)}\|}\leq \rho(T_1),\\
&\limsup\limits_{k\rightarrow\infty}\sqrt[r^k]{\| \widehat{Q}^{(k)}-X_M\|}\leq \rho(T_1)^2,\\
&\limsup\limits_{k\rightarrow\infty}\sqrt[r^k]{\| Q-\widehat{B}^{(k)}-Y_M\|}\leq \rho(T_1)^2.
\end{align*}
\end{subequations}
\begin{example}
For $r=2$, Algorithm~\ref{aa2} becomes the following so-called doubling algorithm, which is R-quadratically convergent,
\begin{align*}
\widehat{A}^{(k+1)}&:=\widehat{A}^{(k)}(\widehat{Q}^{(k)}-\widehat{B}^{(k)})^{-1}\widehat{A}^{(k)},\\
\widehat{B}^{(k+1)}&:=\widehat{B}^{(k)}+\widehat{A}^{(k)}(\widehat{Q}^{(k)}-\widehat{B}^{(k)})^{-1}(\widehat{A}^{(k)})^H,\\
\widehat{Q}^{(k+1)}&:=\widehat{Q}^{(k)}-(\widehat{A}^{(k)})^H (\widehat{Q}^{(k)}-\widehat{B}^{(k)})^{-1} \widehat{A}^{(k)},
\end{align*}
with the initial matrices
$(A^{(1)},B^{(1)},Q^{(1)})$ as given in \eqref{initial2}.
For $r=3$, Algorithm~\ref{aa2} becomes the following so-called tripling algorithm, which is R-cubically convergent,
\begin{align*}
\widehat{A}^{(k+1)}&:=\widehat{A}^{(k)}(\widehat{C}^{(k)})^{-1}\widehat{D}^{(k)},\\
\widehat{B}^{(k+1)}&:=\widehat{B}^{(k)}+\widehat{A}^{(k)}(\widehat{C}^{(k)})^{-1}(\widehat{A}^{(k)})^H,\\
\widehat{Q}^{(k+1)}&:=\widehat{B}^{(k)}+\widehat{C}^{(k)}-(\widehat{D}^{(k)})^H (\widehat{C}^{(k)})^{-1} \widehat{D}^{(k)},\\
\widehat{C}^{(k)}&:=\widehat{Q}^{(k)}-\widehat{B}^{(k)}-(\widehat{A}^{(k)})^H (\widehat{Q}^{(k)}-\widehat{B}^{(k)})^{-1} \widehat{A}^{(k)},\\
\widehat{D}^{(k)}&:=\widehat{A}^{(k)}(\widehat{C}^{(k)})^{-1}\widehat{A}^{(k)},
\end{align*}
with the initial matrices
$(A^{(1)},B^{(1)},Q^{(1)})$ as given in \eqref{initial2}.
\end{example}
\begin{example}
Let $\mathbb{N}$ be the set of natural numbers and let $g$ be a positive integer-valued function on $\mathbb{N}$. Generally, we can modify the Algorithm~\ref{aa2} to the following
iteration according to the group-like property~\eqref{tran1}.
 \begin{Algorithm}\label{aa3}
{\emph{(Accelerative iteration 3 for solving Eqs.~\eqref{eq:NME})}}
\begin{enumerate}
\item \emph{Given a function $g:\mathbb{N}\rightarrow\mathbb{N}$}.
\item  $(\widehat{A}^{(1)},\widehat{B}^{(1)},\widehat{Q}^{(1)})$\quad \mbox{is given in \eqref{initial2}}.
\item    \emph{For} $k=1,2,\ldots,$ \emph{compute}
 \begin{align*}
\widehat{A}^{(k+1)}&=\widehat{A}^{(k)}(\widehat{Q}^{(k)}-B_{g(k)-1}{(k)})^{-1}A_{g(k)-1}{(k)},\\
\widehat{B}^{(k+1)}&=\widehat{B}^{(k)}+\widehat{A}^{(k)}(\widehat{Q}^{(k)}-B_{g(k)-1}{(k)})^{-1}(\widehat{A}^{(k)})^H,\\
\widehat{Q}^{(k+1)}&=Q_{g(k)-1}{(k)}-(A_{g(k)-1}{(k)})^H (\widehat{Q}^{(k)}-B_{g(k)-1}{(k)})^{-1} A_{g(k)-1}{(k-1)},
\end{align*}
    until convergence, where $({A}_{g(k)-1}{(k)},{B}_{g(k)-1}{(k)},{Q}_{g(k)-1}{(k)})$ is defined in step 4.
\item
     \emph{For} $i=1,\cdots,g(k)-2$, iterate
     \begin{align*}
A_{i+1}{(k)}&={A}_{i}{(k)}({Q}_{i}{(k)}-B_{i}{(k)})^{-1}A_{i}{(k)},\\
B_{i+1}{(k)}&={B}_{i}{(k)}+{A}_{i}{(k)}({Q}_{i}{(k)}-B_{i}{(k)})^{-1}{A}_{i}{(k)}^H,\\
Q_{i+1}{(k)}&=Q_{i}{(k)}-A_{i}{(k)}^H ({Q}_{i}{(k)}-B_{i}{(k)})^{-1} A_{i}{(k)},
\end{align*}
with $(A_1{(k)},B_1{(k)},Q_1{(k)})=(\widehat{A}^{(k)},\widehat{B}^{(k)},\widehat{Q}^{(k)})$.
 \end{enumerate}
\end{Algorithm}
Instead of producing the sequence of matrices $({A}^{(k)},B^{(k)},Q^{(k)})$, the
iteration $(\widehat{A}^{(k)},\widehat{B}^{(k)},\widehat{Q}^{(k)})$ produce
 $({A}^{2^{s(k)}},{B}^{2^{s(k)}},{Q}^{2^{s(k)}})$, where $s(k)=\prod\limits_{i=1}^k g(i)$.
 Between Algorithms~\ref{aa2} and \ref{aa3}, the iteration in Algorithm~\ref{aa3} should has a faster rate of
convergence by choosing a suitable $g$. However, it is the most expensive since the enlarged steps of inner iteration may increase the computational cost.
 \end{example}

From \eqref{conv1}, the convergence behaviors in fixed-point iteration \ref{fix1} and two accelerative iterations \eqref{aa1} and \eqref{aa2} are clear when $\rho(T_1)<1$. In the case that $\rho(T_1)=1$ (we refer as critical case), one can show that both algorithms \ref{fix1} and \ref{aa1} converge R-sublinearly to $X_M$. However, there is no further information about the convergence for the Algorithm~\ref{aa2}. The following theory discusses the convergence of the iterative method~\ref{aa2} in the critical case.
\begin{Theorem}\label{DAconvthm}
In the critical case, all sequences generated in Algorithm~\ref{aa2}
for finding the maximal positive definite solution $X_M$ of~\eqref{eq:NME}
are well-defined, provided that the assumptions~\eqref{cond1} corresponding to Eq.~\eqref{eq:NMEP} and $A$ needs to be nonsingular corresponding to Eq.~\eqref{eq:NMEM} are satisfied. Moreover,
\begin{eqnarray*}
A^{(k)}&\rightarrow& 0,\,\mbox{R-linearly}\,\,
\mbox{as } k\rightarrow\infty, \\
Q^{(k)} &\rightarrow & X_M,\,\mbox{R-linearly}\,\,
\mbox{as } k\rightarrow\infty, \\
B^{(k)} &\rightarrow& Q-Y_M,\,\mbox{R-linearly}\,\,
\mbox{as }  k\rightarrow\infty,
\end{eqnarray*}
with convergence rate at least ${1}/\bf{r}$.
\end{Theorem}
\begin{proof}
The proof is analogous to the
\cite{ChiangPHD2009}[Theorem~3.2] and the detailed calculations for the convergence analysis are much tedious.
 We omit it here.
\end{proof}
In the following example we show a scalar result that explains the convergence rates of Algorithm~\ref{fix1}, \ref{aa1} and \ref{aa2} in the critical case.
\begin{example}
Assume that $n = 1$ and $f$ is the identity operator. Then the corresponding equation \eqref{eq:NMEP} can be written as
\begin{eqnarray*}
     x + \frac{|a|^2}{x} = q, 
\end{eqnarray*}
where  $a\in\mathbb{C}$ and the real number $q>0$. We focus on the critical case that $\rho(T_1)=1$, which is equivalent to $q=2|a|$. In this situation, $x_+=x_-=\frac{q}{2}=|a|$ and the Algorithm~\ref{fix1} gives
\begin{eqnarray*}
a_k =  \frac{x_+}{k},\,
b_k = \frac{(k-1)x_+}{k},\,
q_k = \frac{(k+1)x_+}{k}.
\end{eqnarray*}
Similar to Algorithm~\ref{aa1} the following results can be derived,
\begin{eqnarray*}
a_k =  \frac{x_+}{\ell k},\,
b_k = \frac{(\ell{k}-1)x_+}{\ell k},\,
q_k = \frac{(\ell{k}+1)x_+}{\ell k}.
\end{eqnarray*}
The convergence speed of the fixed-point iterations~\ref{fix1} and~\ref{aa1} both are usually very slow. On the other hand, by induction it is easy to see that
\begin{eqnarray*}
a_k =  \frac{1}{r}a_{k-1},\,
b_k-x_+ = \frac{1}{r}(b_{k-1}-x_+), \,
q_k-x_+ = \frac{1}{r}(q_{k-1}-x_+).
\end{eqnarray*}
or
\begin{eqnarray*}
a_k =  \frac{x_+}{r^k},\,
b_k = \frac{(r^k-1)x_+}{r^k},\,
q_k = \frac{(r^k+1)x_+}{r^k}
\end{eqnarray*}
in Algorithm~\ref{aa2}, which coincides with the results in Theorem~\ref{DAconvthm}.
\end{example}
%


\section{An alternating iteration}
As we have seen in the beginning of Section~2, three sequences of matrices $\{A^{(k)}\}$,  $\{B^{(k)}\}$ and $\{Q^{(k)}\}$ are defined by applying $\mathcal{F}_{\pm}^{(2)}$ repeatedly to Eq.~\eqref{kNME} with the help of SMWF. Another possible approach that makes the same result is applying $\mathcal{F}_{\pm}$ twice to Eqs.~\eqref{eq:NME1} . In other words, $\{A^{(k)}\}$,  $\{B^{(k)}\}$ and $\{Q^{(k)}\}$ in \eqref{signal} can be proceeded by strictly alternating between the following two iterations (if exist)
\begin{subequations}\label{signal1}
\begin{align}
A^{(k)}&:=f(A)(f(Q)-B_{1/2}^{(k)})^{-1}A_{1/2}^{(k)},\\
B^{(k)}&:=\pm f(A)(f(Q)-B_{1/2}^{(k)})^{-1}f(A)^H,\\
Q^{(k)}&:=Q_{1/2}^{(k)}\mp(A_{1/2}^{(k)})^H (f(Q)-B_{1/2}^{(k)})^{-1} A_{1/2}^{(k)},
\end{align}
\end{subequations}
and consequently,
\begin{subequations}\label{signal2}
\begin{align}
A_{1/2}^{(k)}&:=A(Q-B^{(k-1)})^{-1}A^{(k-1)},\\
B_{1/2}^{(k)}&:=\pm A(Q-B^{(k-1)})^{-1}A^H,\\
Q_{1/2}^{(k)}&:=Q^{(k-1)}-(A^{(k-1)})^H (Q-B^{(k-1)})^{-1} A^{(k-1)},
\end{align}
\end{subequations}
with the initial matrices $(A_{1/2}^{(1)},B_{1/2}^{(1)},Q_{1/2}^{(1)})=(A,0,Q)$, and any integer $k>1$.
The monotonicity properties of new sequences are shown in
the following theory.
\begin{Theorem}\label{thm51}
We have the following results for each positive integer $k$,
\begin{itemize}
\item[(1)] Consider iterations~\eqref{signal1} and ~\eqref{signal2} corresponding to Eq.~\eqref{eq:NME1} with minus sign. Suppose that the assumption~\eqref{cond1} holds. Let $X_S\in S$, where $S$ is defined in \eqref{cond1}. For the monotonicity of the sequence of matrices $\{B_{1/2}^{(k)}\}$, we have
$B_{1/2}^{(k)}\leq B_{1/2}^{(k+1)}< f(X_S)$. Moreover, we have the following interlacing property,
 \begin{align}\label{ineq}
 f(B_{1/2}^{(k)})\leq B^{(k)}\leq f(B_{1/2}^{(k+1)}).
 \end{align}
For the monotonicity of the sequence of matrices $\{Q_{1/2}^{(k)}\}$, we also have
   \begin{align}\label{ineq1}
   Q_{1/2}^{(k+1)}\leq Q^{(k)}\leq Q_{1/2}^{(k)}.
   \end{align}

\item[(2)] Consider iterations~\eqref{signal1} and ~\eqref{signal2} corresponding to Eq.~\eqref{eq:NME1} with plus sign. For the monotonicity of the sequence of matrices $\{B_{1/2}^{(k)}\}$ we have
$B_{1/2}^{(k+1)}\leq B_{1/2}^{(k)}\leq 0$. Moreover, we have the following interlacing property,
 \begin{align}\label{ineq2}
 B^{(k)}\leq f(B_{1/2}^{(k+1)})\leq f(B_{1/2}^{(k)}).
 \end{align}
For the monotonicity of the sequence of matrices $\{Q_{1/2}^{(k)}\}$, we also have
   \begin{align}\label{ineq3}
   Q_{1/2}^{(k)}\leq Q_{1/2}^{(k+1)}\leq Q^{(k)}.
   \end{align}
\item[(3)]
 For the convergence of sequences of matrices $\{A_{1/2}^{(k)}\}$, $\{B_{1/2}^{(k)}\}$ and $\{Q_{1/2}^{(k)}\}$, we have
 $A_{1/2}^{(k)}\rightarrow 0$, $B_{1/2}^{(k)}\rightarrow A(Q-Y_M)^{-1}A^H$ and $Q_{1/2}^{(k)}\rightarrow X_M$ as $k\rightarrow\infty$.
\end{itemize}
\end{Theorem}
\begin{proof}
\par\noindent
\begin{itemize}
\item[1.]First, $B_{1/2}^{(1)}=0< f(X_S)$ is clear. Since $Q-B^{(k)}-A^H f(X_S)^{-1} A\geq X_S-B^{(k)}>0$, the result $f(X_S)-B_{1/2}^{(k+1)}=f(X_S)-A(Q-B^{(k)})^{-1}A^H>0$ follows from Lemma~~\ref{Schur}. Otherwise, the following two equalities
    \begin{align*}
    f(B_{1/2}^{(k+1)})- B^{(k+1)}=&f(A)[(f(Q)-f(B^{(k)}))^{-1}-(f(Q)-B_{1/2}^{(k+1)})^{-1}]f(A)^H,\\
    f(B^{(k)})- B_{1/2}^{(k+1)}=&A[(Q-f(B_{1/2}^{(k)}))^{-1}-(Q-B^{(k)})^{-1}]A^H,
    \end{align*}
    are directly inspired by iterations~\eqref{signal1} and \eqref{signal2}. It implies that $f(B_{1/2}^{(k+1)})\leq B^{(k+1)}$ if $B^{(k)}\leq f(B_{1/2}^{(k+1)})$ and the last inequality holds if $f(B_{1/2}^{(k)})\leq B^{(k)}$. Thus, the inequality~\eqref{ineq} holds.
    Finally, the inequality~\eqref{ineq1} is guaranteed from the definition of $Q_{1/2}^{(k)}$ and $Q^{(k)}$.
\item[2.] Using similar arguments as in part~(1), we can proof results~\eqref{ineq2} and \eqref{ineq3}.
\item[3.] These convergence results can be verified easily from the iteration~\eqref{signal2}.
\end{itemize}
\end{proof}
From Theorem~\ref{thm51}, the sequence of matrices $(A_{1/2}^{(k)},B_{1/2}^{(k)},Q_{1/2}^{(k)})$ generated by iterations~\eqref{signal1} and~\eqref{signal2} can be carried out with no breakdown. Similar to Proposition~\ref{trans} it is easy to show that iterations~\eqref{signal1} and~\eqref{signal2} also
have the following group-like law property.  The proof is
almost the same as that of the foregoing manner, thus it will be omitted.
\begin{Proposition}\label{trans2}
Suppose that the sequence of matrices $(A_{1/2}^{(k)},B_{1/2}^{(k)},Q_{1/2}^{(k)})$ is generated by iterations \eqref{signal1} and \eqref{signal2}  with initial matrices $(A_{1/2}^{(1)},B_{1/2}^{(1)},Q_{1/2}^{(1)})=(A,0,Q)$. Then, \begin{align*}
A^{(i+j)}&=f(A_{1/2}^{(i)})(f(Q_{1/2}^{(i)})-B_{1/2}^{(j)})^{-1}A_{1/2}^{(j)},\\
B^{(i+j)}&=f(B_{1/2}^{(i)})\pm f(A_{1/2}^{(i)})(f(Q_{1/2}^{(i)})-B_{1/2}^{(j)})^{-1}f(A_{1/2}^{(i)})^H,\\
Q^{(i+j)}&=Q_{1/2}^{(j)}\mp(A_{1/2}^{(j)})^H (f(Q_{1/2}^{(i)})-B_{1/2}^{(j)})^{-1} A_{1/2}^{(j)},
\end{align*}
and
\begin{align*}
A_{1/2}^{(i+j)}&=A_{1/2}^{(i)}(Q_{1/2}^{(i)}-B^{(j)})^{-1}A^{(j)},\\
B_{1/2}^{(i+j)}&=B_{1/2}^{(i)}+ A_{1/2}^{(i)}(Q_{1/2}^{(i)}-B^{(j)})^{-1}(A_{1/2}^{(i)})^H,\\
Q_{1/2}^{(i+j)}&=Q^{(j)}-(A^{(j)})^H (Q_{1/2}^{(i)}-B^{(j)})^{-1} A^{(j)},
\end{align*}
where $i$ and $j$ are any integers.
\end{Proposition}
\begin{Remark}
\par\noindent
\begin{itemize}
\item[1.]
On the following fixed-point iterations,
\begin{subequations}
\begin{align*}
X_k&=Q\mp A^H f(X_{k-1})^{-1}A,\\
X_1&=Q,
\end{align*}
\end{subequations}
we observe that
\begin{align*}
Q^{(k)}=X_{2k},\,Q_{1/2}^{(k)}=X_{2k-1}.
\end{align*}
\item[2.]
From Proposition~\ref{trans2} we have the new representation of $(A_{1/2}^{(k)},B_{1/2}^{(k)},Q_{1/2}^{(k)})$ which only depends on itself, i.e.,
\begin{subequations}
\begin{align*}
A_{1/2}^{(k)}&=A_{1/2}^{(k-1)}(Q_{1/2}^{(k-1)}-B^{(1)})^{-1}A^{(1)},\\
B_{1/2}^{(k)}&=B_{1/2}^{(k-1)}\pm A_{1/2}^{(k-1)}(Q_{1/2}^{(k-1)}-B^{(1)})^{-1}(A_{1/2}^{(k-1)})^H,\\
Q_{1/2}^{(k)}&=Q^{(1)}-(A^{(1)})^H (Q_{1/2}^{(k-1)}-B^{(1)})^{-1} A^{(1)},
\end{align*}
for any integer $k>1$.
\end{subequations}
\end{itemize}
\end{Remark}
 It is interesting to study the convergence rate of iterations~\eqref{signal1} and~\eqref{signal2}.
 The reduction process will need some steps. In the beginning we can easily checked that the matrix equations
 \[
 X\pm (A_{1/2}^{(k)})^H (f(X)-B_{1/2}^{(k)})^{-1}A_{1/2}^{(k)}=Q_{1/2}^{(k)}
 \]
 can be rewritten as
 \[
 \mathcal{M}_{1/2}^{(k)} \left [ \begin{array}{c} I_n \\ X_M \end{array}\right ] = \mathcal{L}_{1/2}^{(k)}
\left [ \begin{array}{c} I_n \\ {f(X_M)} \end{array}\right ]T^{(k)}_{1/2},
 \]
 where
 \begin{equation}\label{mlh}
   \mathcal{M}_{1/2}^{(k)}:=
   \left [
   \begin{array}{rc} A_{1/2}^{(k)} & 0 \\ Q_{1/2}^{(k)} & -I_n
   \end{array} \right ] , \quad
   \mathcal{L}_{1/2}^{(k)}:= \left [ \begin{array}{lc} -B_{1/2}^{(k)} & I_n \\ \pm (A_{1/2}^{(k)})^H & 0 \end{array}
   \right ],
\end{equation}
and $T_{1/2}^{(k)}=(f(X_M)-B_{1/2}^{(k)})^{-1}A_{1/2}^{(k)}$.
As discussed before, the convergence speed of iterations~\eqref{signal1} and~\eqref{signal2} is highly related to the magnitude of $\|T_{1/2}^{(k)}\|$. More precisely, \eqref{mlh} gives rise to the estimation of the error bound for this iterative method.
\begin{align}\label{esti}
\|Q_{1/2}^{(k)}-X_M\|=\|(A^{(k-1)})^H(Q-B^{(k-1)})^{-1}A^H T_{1/2}^{(k)}\|\leq\frac{\|X_M-Y_M\|\|T^{(1)}\|^k}{\|Q\|-\|Y_M\|} \|T_{1/2}^{(k)}\|.
\end{align}
It is natural to ask whether the spectral radius of $T_{1/2}^{(1)}$ is less than or equal to one and $T_{1/2}^{(k)}=(T_{1/2}^{(1)})^k$. Unfortunately, the answer is no. Consider a conjugate NME with coefficient matrices $A$ and $Q$ given by
\begin{align*}
A=\bb  26i & -16+2i\\
 -14+9i &  -19-9i \eb,\,
Q=\bb 128.193 & 24.813+92.180i\\
    24.813-92.180i & 97.003\eb,
\end{align*}
which is generated randomly by Matlab. Then the maximal positive definite solution $X_M=\bb  120.595 & 28.387+85.261i \\
 28.387-85.261i & 80.758\eb$ and we see that $\rho(T_{1/2}^{(1)})=1.222>1.042=\rho(T_{1/2}^{(2)})>1$. However,
it might be interesting to investigate the relationship between $T^{(k)}$ and $T_{1/2}^{(k)}$, which we state below.
\begin{Theorem}\label{T}
For any positive integers $i,j$ and $k$, we have
$$T_{i+j-1}=f(T_{1/2}^{(i)})T_{1/2}^{(j)},\,\,T_{1/2}^{(i+j)}=T_{1/2}^{(i)}T_j,$$
 and we conclude that
 $$T_{1/2}^{(i+j+k-1)}=T_{1/2}^{(k)} f(T_{1/2}^{(i)})T_{1/2}^{(j)}.$$
 Specially, $T_{1/2}^{(k)}=T_{1/2}^{(1)}T_1^{k-1}$ . 
\end{Theorem}
\begin{proof}
Let
\begin{align*}
\mathcal{M}^{(i,j)}_{\star} &:= \left [
\begin{array}{rc} f(A_{1/2}^{(i)}) (f(Q_{1/2}^{(i)}) -B_{1/2}^{(j)} )^{-1} & 0
\\ \mp(A_{1/2}^{(j)} )^H (f(Q_{1/2}^{(i)}) -B_{1/2}^{(j)})^{-1} & I_n \end{array} \right ],\,\\
\mathcal{L}^{(i,j)}_{\star} &:=
\left [
\begin{array}{cr}
I_n & -f(A_{1/2}^{(i)})(f(Q_{1/2}^{(i)}) -B_{1/2}^{(j)})^{-1} \\ 0 & \pm(A_{1/2}^{(j)} )^H (f(Q_{1/2}^{(i)}) -B_{1/2}^{(j)})^{-1} \end{array}
\right ].
\end{align*}
And
\begin{align*}
\mathcal{M}^{(i,j)}_{1/2,\star} &:= \left [
\begin{array}{rc} A_{1/2}^{(i)} (Q_{1/2}^{(i)} -B^{(j)} )^{-1} & 0
\\ \mp(A_{1/2}^{(j)} )^H (Q_{1/2}^{(i)} -B^{(j)})^{-1} & I_n \end{array} \right ],\,\\
\mathcal{L}^{(i,j)}_{1/2,\star} &:=
\left [
\begin{array}{cr}
I_n & -A_{1/2}^{(i)}(Q_{1/2}^{(i)} -B^{(j)})^{-1} \\ 0 & \pm(A_{1/2}^{(j)} )^H (Q_{1/2}^{(i)} -B^{(j)})^{-1} \end{array}
\right ].
\end{align*}
Then, $\mathcal{M}^{(i,j)}_{\star}\mathcal{L}_{1/2}^{(j)}=\mathcal{L}^{(i,j)}_{\star}f(\mathcal{M}_{1/2}^{(i)})$ and
$\mathcal{M}^{(i,j)}_{1/2,\star}\mathcal{L}^{(j)}=\mathcal{L}^{(i,j)}_{\star}\mathcal{M}_{1/2}^{(i)}$.
Further,
$\mathcal{M}^{(i+j-1)}=\mathcal{M}^{(i,j)}_{\star}\mathcal{M}_{1/2}^{(j)},\, \mathcal{L}^{(i+j-1)}=\mathcal{L}^{(i,j)}_{\star}f(\mathcal{L}_{1/2}^{(i)})$,
and
$\mathcal{M}_{1/2}^{(i+j)}=\mathcal{M}^{(i,j)}_{1/2,\star}\mathcal{M}^{(j)},$  $\mathcal{L}_{1/2}^{(i+j)}=\mathcal{L}^{(i,j)}_{1/2,\star}\mathcal{L}_{1/2}^{(i)}$. We are now in a position to present this result, as the following comparison with both sides,
\begin{align*}
&\mathcal{M}^{(i+j-1)}\left [ \begin{array}{c} I_n \\ X_M \end{array}\right] =\mathcal{M}^{(i,j)}_{\star}\mathcal{M}_{1/2}^{(j)}\left [ \begin{array}{c} I_n \\ X_M \end{array}\right ]=
\mathcal{M}^{(i,j)}_{\star}\mathcal{L}_{1/2}^{(j)}\left [ \begin{array}{c} I_n \\ f(X_M) \end{array}\right ] T_{1/2}^{(j)}\\
&=\mathcal{L}_{\star}^{(i,j)}f(\mathcal{M}_{1/2}^{(i)})\left [ \begin{array}{c} I_n \\ f(X_M) \end{array}\right ] T_{1/2}^{(j)}=\mathcal{L}_{\star}^{(i,j)}f(\mathcal{L}_{1/2}^{(i)})\left [ \begin{array}{c} I_n \\ X_M \end{array}\right ] f(T_{1/2}^{(i)})T_{1/2}^{(j)}\\
&=\mathcal{L}^{(i+j-1)}\left [ \begin{array}{c} I_n \\ X_M \end{array}\right ] f(T_{1/2}^{(i)})T_{1/2}^{(j)}.
\end{align*}
Thus $T_{i+j-1}$ is equal to $f(T_{1/2}^{(i)})T_{1/2}^{(j)}$. Next,
\begin{align*}
&\mathcal{M}_{1/2}^{(i+j+k-1)}\left [ \begin{array}{c} I_n \\ X_M \end{array}\right] =\mathcal{M}^{(k,i+j-1)}_{1/2,\star}\mathcal{M}^{(i+j-1)}\left [ \begin{array}{c} I_n \\ X_M \end{array}\right ]\\
&=\mathcal{M}^{(k,i+j-1)}_{1/2,\star}\mathcal{L}^{(i+j-1)}\left [ \begin{array}{c} I_n \\ X_M \end{array}\right ] f(T_{1/2}^{(i)})T_{1/2}^{(j)}=\mathcal{L}^{(k,i+j-1)}_{1/2,\star}\mathcal{M}_{1/2}^{(k)}\left [ \begin{array}{c} I_n \\ X_M \end{array}\right ] f(T_{1/2}^{(i)})T_{1/2}^{(j)}\\
&=\mathcal{L}_{1/2}^{(i+j+k-1)}\left [ \begin{array}{c} I_n \\ f(X_M) \end{array}\right ] T_{1/2}^{(k)}f(T_{1/2}^{(i)})T_{1/2}^{(j)}.
\end{align*}
This completes the proof.
\end{proof}
From Theorem~\ref{T} and Eq.~\eqref{esti}, we know that
\[
\limsup\limits_{k\rightarrow\infty}\sqrt[k]{\| Q_{1/2}^{(k)}-X_M\|}\leq \rho(T_1)^2.
\]
When $\rho(T^{(1)})<1$, $Q_{1/2}^{(k)}$ converges R-linearly to $X_M$ with rate $\rho(T_1)^2$.
Otherwise, we can establish the sublinear convergence property for $Q_{1/2}^{(k)}$ by using of the same technique in \cite{ChiangPHD2009}. This issue is not discussed further here.
\section{Concluding Remark}
In this paper, we investigate the positive definite solutions of a class of nonlinear matrix equations.
Taking advantage of some famous transformations, the structure of this equation is still preserved. Under some certain conditions, it is proved that the maximum positive definite solution of Eq.~\eqref{eq:NMEP} (or Eq.~\eqref{eq:NMEM}) coincides with the maximum positive definite solution of the standard nonlinear matrix equation~\eqref{NMEP}. In addition, an iterative method with R-superlinear with order $r>1$ for solving the maximum positive definite solution of the equation has been investigated considerably based on a fixed-point iteration. The techniques of the Proposition~\ref{trans} can be employed in the convergence analysis of this acceleration of iterative method. An interesting issue is how many iterations like the form~\eqref{fix0} satisfying the group-like law property~\eqref{ijk}. This will be further explored in the future.

  \section*{Acknowledgment}
The author wishes to thank Dr. Ying-Ju Tessa Chen (Department of Information Systems and Analytics, Miami University) for many interesting and valuable suggestions on the manuscript. This research work is partially supported by the Ministry of Science and Technology and the National Center for Theoretical Sciences in Taiwan.

\bibliographystyle{plain}

\begin{thebibliography}{10}

\bibitem{Anderson90}
W.~N. Anderson, T.~D. Morley, and G.~E. Trapp.
\newblock Positive solutions to {$X=A-BX^{-1}B^*$}.
\newblock {\em Linear Algebra Appl.}, 134:53--62, 1990.

\bibitem{Bernstein2005}
D.~S. Bernstein.
\newblock {\em Matrix mathematics:Theory, facts, and formulas with application
  to linear systems theory}.
\newblock Princeton University Press, Princeton, NJ, 2005.

\bibitem{Bini2012}
D.~A. Bini, B.~Iannazzo, and B.~Meini.
\newblock {\em Numerical solution of algebraic {R}iccati equations}, volume~9
  of {\em Fundamentals of Algorithms}.
\newblock Society for Industrial and Applied Mathematics (SIAM), Philadelphia,
  PA, 2012.

\bibitem{emrl2006364}
P.~$\check{\mbox{S}}$emrl.
\newblock Maps on matrix spaces.
\newblock {\em Linear Algebra Appl.}, 413(2–3):364 -- 393, 2006.

\bibitem{Chiang2015}
C.-Y. Chiang.
\newblock On the {S}ylvester-like matrix equation ${AX+f(X)B=C}$.
\newblock (In press).

\bibitem{ChiangThesis08}
C.-Y. Chiang.
\newblock {\em Convergence Analysis of the Structure-Preserving Doubling
  Algorithms for Nonlinear Matrix Equations}.
\newblock PhD thesis, Department of Mathematics, National Tsing Hua University,
  Hsinchu, Taiwan, July 2008.

\bibitem{ChiangPHD2009}
C.-Y. Chiang, E.~K.-W. Chu, C.-H. Guo, T.-M. Huang, W.-W. Lin, and S.-F. Xu.
\newblock Convergence analysis of the doubling algorithm for several nonlinear
  matrix equations in the critical case.
\newblock {\em SIAM J. Matrix Anal. Appl.}, 31(2):227--247, 2009.

\bibitem{Chu04}
E.~K.-W. Chu, H.-Y. Fan, and W.-W. Lin.
\newblock {S}tructure-preserving algorithms for periodic discrete-time
  algebraic {R}iccati equations.
\newblock {\em Int. J. Control}, 77:767--788, 2004.

\bibitem{Chu05}
E.~K.-W. Chu, H.-Y. Fan, and W.-W. Lin.
\newblock Structure-preserving algorithms for continuous-time algebraic
  {R}iccati equations.
\newblock {\em Linear Algebra Appl.}, 396:55--80, 2005.

\bibitem{Sayed01}
Salah~M. El-Sayed and Andr$\acute{\mbox{e}}$ C.~M. Ran.
\newblock On an iteration method for solving a class of nonlinear matrix
  equations.
\newblock {\em SIAM J. Matrix Anal. Appl.}, 23(3):632--645, 2002.

\bibitem{Engweradaran93}
J.C. Engwerda, A.~C.~M. Ran, and A.L. Rijkeboer.
\newblock Necessary and sufficient conditions for the existence of a positive
  definite solution of the matrix equation {$X+A^\ast X^{-1}A=Q$}.
\newblock {\em Linear Algebra Appl.}, 186:255--274, 1993.

\bibitem{Fang2006601}
L.~Fang and G.~Ji.
\newblock Linear maps preserving products of positive or {H}ermitian matrices.
\newblock {\em Linear Algebra Appl.}, 419(2 - 3):601 -- 611, 2006.

\bibitem{Guo_01SIMAA}
C.-H. Guo.
\newblock Convergence rate of an iterative method for a nonlinear matrix
  equation.
\newblock {\em SIAM J. Matrix Anal. Appl.}, 23(1):295--302, 2001.

\bibitem{GuoLan_99_MC}
C.-H. Guo and P.~Lancaster.
\newblock Iterative solution of two matrix equations.
\newblock {\em Math. Comp.}, 68:1589--1603, 1999.

\bibitem{Hwang07}
T.-M. Huang and W.-W. Lin.
\newblock Structured doubling algorithms for weakly stabilizing {H}ermitian
  solutions of algebraic {R}iccati equations.
\newblock {\em Linear Algebra Appl.}, 430(5-6):1452--1478, 2009.

\bibitem{Hwang05}
T.-M. Hwang, E.~K.-W. Chu, and W.-W. Lin.
\newblock A generalized structure-preserving doubling algorithm for generalized
  discrete-time algebraic {R}iccati equations.
\newblock {\em Int. J. Control}, 78(14):1063--1075, 2005.

\bibitem{10.2307/4100245}
I.~G. Ivanov, V.~I. Hasanov, and F.~Uhlig.
\newblock Improved methods and starting values to solve the matrix equations
  ${X \pm A^{*} X^{-1} A = I}$ iteratively.
\newblock {\em Math. Comp.}, 74(249):263--278, 2005.

\bibitem{Kelley95}
C.~T. Kelley.
\newblock {\em Iterative Methods for Linear and Nonlinear Equations}.
\newblock Number~16 in Frontiers in Applied Mathematics. Society for Industrial
  and Applied Mathematics (SIAM), Philadelphia, PA, 1995.

\bibitem{Li2014546}
Z.-Y. Li, B.~Zhou, and J.~Lam.
\newblock Towards positive definite solutions of a class of nonlinear matrix
  equations.
\newblock {\em Appl. Math. Comput.}, 237:546 -- 559, 2014.

\bibitem{Lin06}
W.-W. Lin and S.-F. Xu.
\newblock Convergence analysis of structure-preserving doubling algorithms for
  {R}iccati-type matrix equations.
\newblock {\em SIAM J. Matrix Anal. Appl.}, 28(1):26--39, 2006.

\bibitem{Meini_02_MC}
B.~Meini.
\newblock Efficient computation of the extreme solutions of {$X + A^* X^{-1} A
  = Q$ and $X - A^* X^{-1} A = Q$}.
\newblock {\em Math. Comp.}, 71:1189--1204, 2002.

\bibitem{doi:10.1137/1.9780898719468}
J.~Ortega and W.~Rheinboldt.
\newblock {\em Iterative Solution of Nonlinear Equations in Several Variables}.
\newblock Society for Industrial and Applied Mathematics (SIAM), Philadelphia,
  PA, 2000.

\bibitem{Ran200215}
Andr$\acute{\mbox{e}}$ C.~M. Ran and Martine~C.B. Reurings.
\newblock On the nonlinear matrix equation ${X+A^\star\mathcal{F}(X)A=Q}$:
  solutions and perturbation theory.
\newblock {\em Linear Algebra Appl.}, 346(1--3):15 -- 26, 2002.

\bibitem{roman2005field}
S.~Roman.
\newblock {\em Field Theory}.
\newblock Graduate Texts in Mathematics. Springer New York, 2005.

\bibitem{Zhan_96_SISC}
X.~Zhan.
\newblock Computing the extremal positive definite soluions of a matrix
  equation.
\newblock {\em SIAM J. Sci. Comput.}, 17:1167--1174, 1996.

\bibitem{Zhou20137377}
B.~Zhou, G.-B. Cai, and J.~Lam.
\newblock Positive definite solutions of the nonlinear matrix equation.
\newblock {\em Appl. Math. Comput.}, 219(14):7377 -- 7391, 2013.

\end{thebibliography}

\end{document}